\documentclass[a4paper]{article}
\usepackage{todonotes}
\usepackage[export]{adjustbox}
\usepackage[english]{babel}
\usepackage[utf8x]{inputenc}
\usepackage{amsmath}
\usepackage{mathtools}
\usepackage{amssymb}
\usepackage{graphicx}
\usepackage{hyperref}
\usepackage{amsthm}
\usepackage{framed}
\usepackage{bm}
\usepackage{eucal}
\usepackage{color}
\usepackage{caption}
\usepackage{subcaption}
\numberwithin{equation}{section}
\usepackage{geometry}\usepackage{mdframed}
\newtheorem{theorem}{Theorem}
\newtheorem{corollary}{Corollary}

\title{Online Estimation and Adaptive Control for a Class of  History Dependent Functional Differential Equations}
\author{ Shirin Dadashi \thanks{sdadashi@vt.du, Graduate Student Department of Mechanical Engineering, Virginia Tech}, Parag Bobade\thanks{paragb4@vt.edu, Graduate Student, Department of Engineering Science and Mechanics, Virginia Tech}, and Andrew J. Kurdila\thanks{kurdila@vt.edu, W.Martin Johnson Professor, Department of Mechanical Engineering, Virginia Tech}}
\date{}
\begin{document}
\maketitle

\begin{abstract}
This paper presents sufficient conditions for the convergence of online estimation methods and  the stability of  adaptive control strategies for a class of history dependent, functional differential equations. The study is motivated by the increasing interest in  estimation and control techniques for  robotic systems whose governing equations include history dependent nonlinearities. The  functional differential equations in this paper are constructed using integral operators that depend on distributed parameters. As a consequence the resulting estimation and control equations  are examples of distributed parameter systems whose states and distributed parameters evolve in finite and infinite dimensional spaces, respectively.    Well-posedness, existence, and uniqueness are discussed for the class of  fully actuated robotic systems with history dependent forces in their governing equation of motion.  By deriving rates of approximation for the class of history dependent operators in this paper, sufficient conditions are derived that guarantee that finite dimensional approximations of the online estimation equations converge to the solution of the infinite dimensional, distributed parameter system.  The convergence and stability of a sliding mode adaptive control strategy for the history dependent, functional differential equations is established using Barbalat's lemma.
\end{abstract}
\section{Introduction}
\label{sec:intro}
It is typical in  texts that introduce the fundamentals of modeling, stability, and control of robotic systems to assume that the underlying governing equations consist of a set of coupled nonlinear ordinary differential equations. This is a natural assumption when methods of analytical mechanics are used to derive the governing equations for systems composed of rigid bodies connected by ideal joints. A quick perusal of the textbooks \cite{shv2005}, \cite{ssvo2010}, or \cite{lda2004}, for example, and the references therein gives a good account of the diverse collection of approaches that have been derived for this class of robotic system over the past few decades. Theses methods have been subsequently refined by numerous authors.  Over roughly the same period, the technical community has shown a continued interest in systems that are governed by nonlinear, functional differential equations. These methods that helped to define the direction of initial efforts in the study of well-posedness and stability  include  \cite{m1949}, \cite{k1956a},\cite{k1956b}, and their subsequent development is expanded in \cite{driver1962}, \cite{rudakov1974}, \cite{rudakov1978}.  More recently, specific  control strategies for classes of functional differential equations have appeared in \cite{rs2002}, \cite{irs2002}, and  \cite{ilr2010}. 
The research described in some cases above deals with quite general plant models. These can include classes of delay equations and general history dependent nonlinearities. One rich collection of history dependent models includes hysteretically nonlinear systems.  General discussions of nonlinear  hysteresis models can be found in \cite{visintin1994} or \cite{bs1996}, and some authors have studied the convergence and stability of systems with nonlinear hysteresis.  For example,   a synthesis of controllers for single-input / single-output functional differential equations is presented  in \cite{rs2002} and \cite{ilr2010}, and these efforts include a wide class of scalar hysteresis operators. 

The success of adaptive control strategies in classical manipulator robotics, as exemplified by \cite{shv2005}, \cite{lda2004}, \cite{ssvo2010}, can be attributed to a large degree to the highly structured form of the governing system of nonlinear ordinary differential equations.  As is well-known, much of the body of work in adaptive control for robotic systems relies on traditional linear-in-parameters assumptions. 

The purpose of this paper is to explore the degree to which the approaches that have been so fruitful in adaptive control of robotic manipulators can be extended to robotic systems governed by certain history dependent, functional differential equations. Emulating the strategy used for robotic systems modeled by ordinary differential equations, we restrict attention to a class of hysteresis operators that satisfy a linear in {\em distributed} parameters condition. That is, the contribution to the functional differential equations takes the form of a nonlinear, history dependent operator that acts linearly on an  infinite dimensional and unknown distributed parameter. 

We illustrate the class of models that are considered in this paper by outlining a variation on two familiar problems encountered in robotic manipulator dynamics, estimation, and control.  Consider the task of developing a model and synthesizing a controller for a flapping wing, test robot that will be used to study aerodynamics in a wind tunnel.  See \cite{bsb2013} for such a system that has been developed by researchers at Brown University over the past few years.   Dynamics for a ground based flapping wing robot can be derived using analytical mechanics in a formulation that is tailored to the structure of a serial kinematic chain \cite{lda2004}, \cite{shv2005}, \cite{ssvo2010}.  The equations of motion take the form  
\begin{equation}
M(q(t))\ddot{q}(t)+C(q(t),\dot{q}(t))\dot{q}(t)+\frac{\partial{V}}{\partial{q}}=Q_{a}(t,\mu) + {\tau}(t)
\label{eq:robotics}
\end{equation}
where $M(q)(t){\in}{\mathbb{R}}^{N{\times} N}$ is the generalized inertia or mass matrix, $C(q(t),\dot{q}(t)){\in}{\mathbb{R}}^{N{\times} N}$ is a nonlinear matrix that represents Coriolis and centripetal contributions, $V$ is the potential energy, $Q_{a}(t,\mu){\in}{\mathbb{R}}^{N}$ is a  vector of generalized aerodynamic forces,  and ${\tau(t)}{\in}{\mathbb{R}}^{N}$ is the actuation force or torque vector. The generalized forces $Q(t,\mu)$ due to aerodynamic loads are assumed to be expressed in terms of history dependent operators that are carefully  discussed below in Section \ref{sec:history}, and $\mu$ is the { \em distributed parameter}  that defines the specific history dependent operator.   For the current discussion, it suffices to note that the aerodynamic contributions are unknown, nonlinear, unsteady,  and notoriously difficult to characterize.   

We consider two specific sets of equations in this paper that are derived from the robotic Equations \ref{eq:robotics}, both of which have similar form.  We are interested in online identification problems in which we seek to find the final state and distributed parameters from observations of the states of the evolution equation.  We are also interested in control synthesis where we choose the input to drive the system to some desired configuration, or to track a given input trajectory.  To simplify our discussion, and following the standard practice for many control synthesis problems for robotics, we choose the original control input to be a partial feedback linearizing control that that reformats the control problem in a standard form.  In the case of online identification, we choose  the input $\tau=M(q)(u-G_1 \dot{q} - G_0 q)-(C(q,\dot{q})\dot{q} + \frac{\partial V}{\partial q}(q)) $  so that the governing equations take the form 
\begin{equation}
\label{eqn:2}\frac{d}{dt}
\begin{bmatrix}
q(t)\\
\dot{q}(t)
\end{bmatrix}
=
\begin{bmatrix}
0&I\\
-G_0&-G_1
\end{bmatrix}
\begin{bmatrix}
q(t)\\
\dot{q}(t)
\end{bmatrix}
+
\begin{bmatrix}
0 \\
I
\end{bmatrix}
(M^{-1}(q)Q_a(t,\mu) +u(t)).
\end{equation}
in terms of a new input $u$. 
The goal in the online identification problem is to learn the parameters $\mu$ and limiting values $q_\infty,\dot{q}_\infty$ from knowledge of the inputs and states $(u,q,\dot{q})$. 
We are also interested in tracking control problems.  When the desired trajectory is given by $q_d$,   we choose the input $\tau=M(q)(u+\ddot{q}_d-G_1 \dot{e} - G_0 e)-(C(q,\dot{q})\dot{q} + \frac{\partial V}{\partial q}(q))$,and the equations governing the tracking error $e:=q-q_d$ take the form
\begin{equation}
\label{eqn:3}\frac{d}{dt}
\begin{bmatrix}
e(t)\\
\dot{e}(t)
\end{bmatrix}
=
\begin{bmatrix}
0&I\\
-G_0&-G_1
\end{bmatrix}
\begin{bmatrix}
e(t)\\
\dot{e}(t)
\end{bmatrix}
+
\begin{bmatrix}
0 \\
I
\end{bmatrix}
(M^{-1}(e+q_d) Q_a(t,\mu)+u(t))
\end{equation}
In either  of the above two cases, we will show in the next section that the equations  can be written in the general form 
\begin{equation}
\dot{X}(t)=AX(t)+B((\mathcal{H}X)(t)\circ \mu +u(t)).
\label{eq:1st_order}
\end{equation}
where ${A}\in \mathbb{R}^{m\times m}$ is the system matrix, ${B}\in \mathbb{R}^{m\times q}$ is the control input matrix, ${u}(t)\in \mathbb{R}^{q}$ is the corresponding input, and ${(\mathcal{H}X)}(t)$ is a history dependent operator that acts on the distributed parameter ${\mu}$.

\section{History Dependent Operators}
\label{sec:history}
There is a significant body of research to model and study the unsteady aerodynamic phenomena in flapping flight. Many different models have been presented in the last twenty years to study the aerodynamics and control of flapping flight. Numerically intensive computational fluid dynamics (CFD) presents a precise method to simulate and study the unsteady lift and drag aerodynamic forces. Generally CFD methods exploit high dimensional models that incorporate computationally expensive moving boundary techniques for the Navier-Stokes equations. They are powerful tools to explain some of the characteristics of the aerodynamic forces. One of the characteristics that has inspired the approach here is the history dependence of the aerodynamic lift and drag functions. We refer the interested reader to \cite{Dadashi2016} to study this phenomena in detail. Although CFD methods are advantages in several aspects, they suffer from curse of dimensionality which makes them a very unfavorable choice for online control applications. In this section, we model the unsteady aerodynamics using history dependent operators. Moreover, we present a method that provides an alternative to a high dimensional aerodynamic model that typically evolves in a much lower dimensional space. We also study the accuracy of the presented method with respect to the resolution level of the lower dimensional model.

\subsection{A Class of History Dependent Operators}
Methods for modeling history dependent nonlinearities can be formulated using a wide array of approaches. Analytical methods for the study of such systems can be based on ordinary  or partial differential equations, differential inclusions, functional differential equations, delay differential equations, or operator theoretic approaches. See references \cite{kp89},\cite{visintin1994},\cite{shv2005}. This paper  treats evolution equations that are constructed using a specific class of history dependent operators $\mathcal{H}$  that are defined in terms of integral operators constructed from history dependent kernels.  These operators are studied in general in \cite{kp89} and \cite{visintin1994}. In this paper the history dependent operators are mappings 
\[ 
\mathcal{H}:C([0,T),\mathbb{R}^{m}){\rightarrow}
C([0,T),P^*)  
\]
where the $T$ is the final time of an interval under consideration, $m$ is the number of input functions, $q$ is the number of output functions, $P$ is a Hilbert space of distributed parameters and its topological dual space $P^*$. We limit our consideration to input$-$output relationships that take the form
\begin{equation}
y(t)=(\mathcal{H}X)(t){\circ}{\mu}
\label{eq:2.1}
\end{equation}
for each $t\in [0,T)$ where 
$y(t) \in \mathbb{R}^q$,$(\mathcal{H}X)(t) \in P^*$, and $\mu \in P $. 

The definition of $\mathcal{H}$ in this paper is carried out in several steps. 
All of our history dependent operators $\mathcal{H}$ are defined by a superposition or  weighting of elementary hysteresis kernels $\kappa_i$ that are continuous as mappings $\kappa_i :\Delta \times [0,T) \times C[0,T) \rightarrow C[0,T)$ for $i=1,\ldots,\ell$.   We first define the operator $h_i:C[0,T) \rightarrow C([0,T),P^*)$ 
\begin{equation}
 (h_if)(t) \circ \mu_i := \iint_{\Delta} \kappa_i (s,t,f) \mu_i(s) ds 
 \label{eq:hys_int_op}
\end{equation}
for $\mu_i \in P_i$ and $P=P_1 \times \cdots \times P_\ell$. 
When we consider problems such as in our motivating examples and numerical case studies, we must construct vectors $H$ of  history dependent operators   where we  define the diagonal matrix
$$
(H X)(t) := \begin{Bmatrix} h_1(a(X))(t) & & 0 \\ &\ddots & \\0 & & h_\ell(a(X))(t)  \end{Bmatrix}
$$
for each $t\in [0,T)$ where $a : \mathbb{R}^m \rightarrow  \mathbb{R}$ is some nonlinear smooth map. Finally, our applications to robotics require that we consider 
\begin{equation}
(\mathcal{H}X)(t) = b(X(t))(HX)(t),
\label{eq:hdef1}
\end{equation}
where $b:\mathbb{R}^m \rightarrow \mathbb{R}^{q\times \ell}$ is some nonlinear, smooth map. 
In terms of our entrywise definitions of the input--output mappings, we have
\begin{equation}
y_i(t):= \sum_{j=1}^{l} b_{ij}(X(t))[h_j(a(X))](t)\circ \mu_j
\label{eq:hdef2}
\end{equation}
for $i=1,\cdots, q$. 

In the following discussion, let $\kappa$ be a generic representation of any of the kernels $\kappa_i$ for $i=1,\ldots, \ell$. We choose a typical kernel  $\kappa(s,t,f)$ to be a special case of a  generalized play operator \cite{visintin1994}.  We suppose that $f$ is a piecewise linear function on $[0,t]$ with breakpoints $0=t_0 < t_1 < \cdots < t_N=t$. The output function  $t \mapsto \kappa(s,t,f)$, for a fixed $s=(s_1,s_2) \in \Delta \subset \mathbb{R}^2$ and piecewise linear $f:[0,t] \rightarrow \mathbb{R}$,     is defined 
by the  recursion where $\kappa^{n-1}:=\kappa(s,t_{n-1},f)$ and for $t \in[t_{n-1},t_n]$ we have 
\begin{align*}
\kappa(s,t,f):= \left \{
\begin{array}{ccc}
\max \left \{ \kappa^{n-1}, \gamma_{s_2}(f(t))  \right \} & & f \text{ increasing on }  [t_{n-1},t_n], \\
\min \left \{  \kappa^{n-1}, \gamma_{s_1}(f(t))  \right\} & & f \text{ decreasing on } [t_{n-1},t_n].
\end{array}
\right .
\end{align*}
The recursion above depends on the choice of   the left and right bounding functions $\gamma_{s_1},\gamma_{s_2}$ that are depicted in Figure \ref{fig:fig_play1}. There are given in terms of a single ridge function $\gamma:\mathbb{R} \rightarrow \mathbb{R}$  with
\begin{align}
\gamma_{s_2}(\cdot)&:=\gamma(\cdot-s_2), \notag \\
\gamma_{s_1}(\cdot)&:=\gamma(\cdot-s_1). \label{eq:defgamma}
\end{align}

\begin{figure}[h!]
\centering
\includegraphics[width=.45\textwidth]{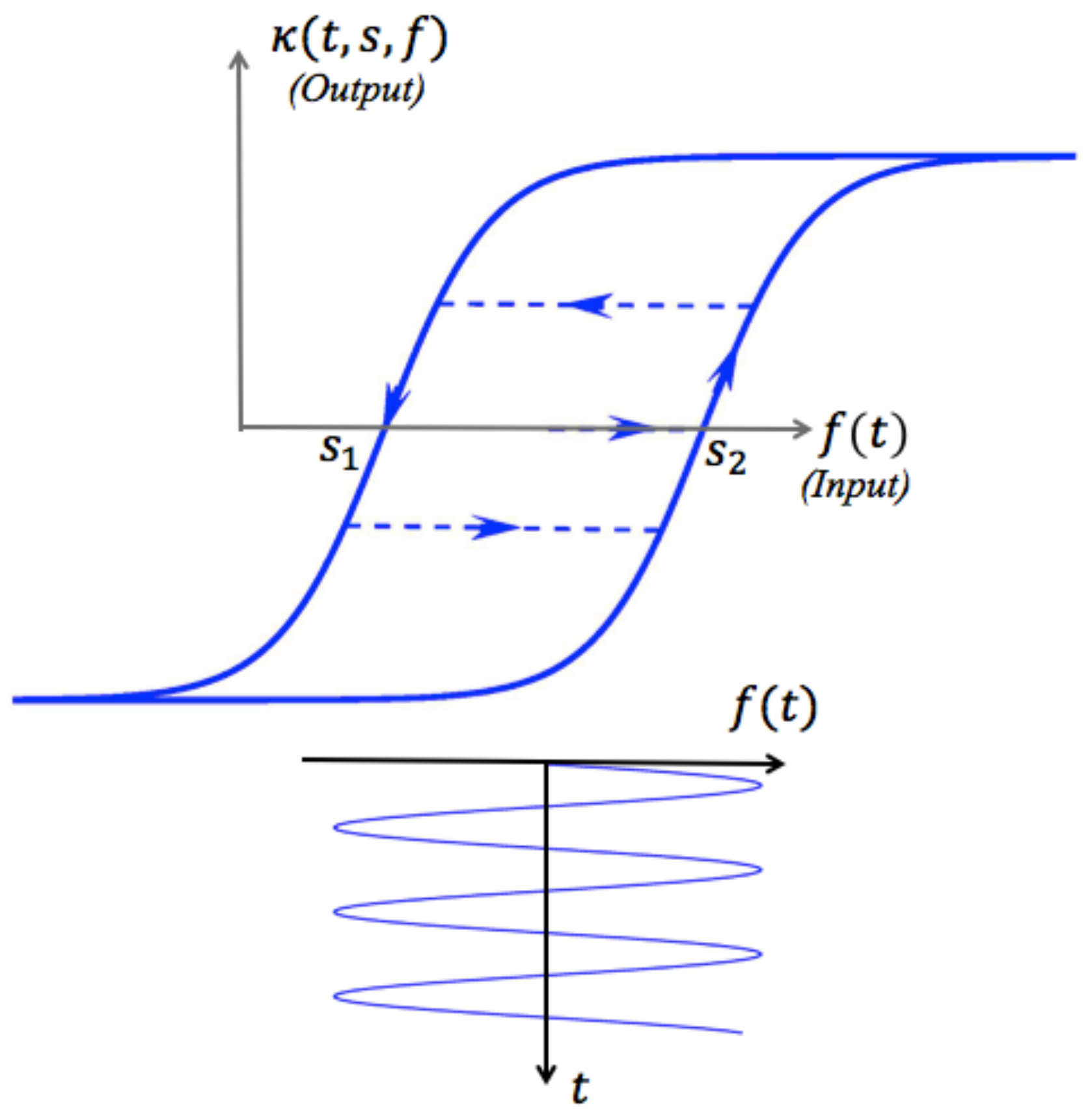}
\caption{Elementary hysteresis kernel $t \rightarrow \kappa(s,t,f)$ for fixed $s=(s_1,s_2)\in \mathbb{R}^2$ and piecewise continuous $f:[0,t) \rightarrow \mathbb{R}$.}
\label{fig:fig_play1}
\end{figure}

As noted in \cite{visintin1994}, the definition of $\kappa$ is extended for any $f\in C[0,T)$ by a continuity and density argument.  

\subsection{Approximation of History Dependent Operators}
\label{subsec:approx_hist}
The integral operator introduced in Equation \ref{eq:hys_int_op} allows for the representation of complex hysteretic response via the superposition or weighting of fundamental kernels $\kappa_i$.  These fundamental kernels, each of which has simple input-output relationships,  play the role of building blocks for modeling much more complex response characteristics. See  \cite{ksw2003} for  studies of  history dependent active materials,  \cite{Dadashi2016} for applications that represent nonlinear aerodynamic loading, or  Section \ref{sec:numerical} of this paper to see an example of richness of this class of models.  In this section we emphasize another important feature of this particular class of history dependent operators. We show that relatively simple approximation methods yields bounds on the error in approximation of the history dependent operator that are uniform in time and over the class of functions $\mu\in P$. 

\subsection{Approximation Spaces $\mathcal{A}^\alpha_2$}
The approximation framework we follow in this paper is based on a straightforward implementation of approximation spaces discussed in detail in   \cite{dl1993} or   \cite{devore1998},  and further  developed by Dahmen in \cite{dahmen1996}. We will see that approximation of the class of history dependent operators under consideration exploit a well-known connection between the class of Lipschitz functions and certain approximation spaces as described in \cite{devore1998}.

\subsection{Wavelets and Approximation Spaces}
Multiresolution Analysis ( MRA )  techniques use results from wavelet theory to model multiscale phenomena. To motivate our discussion, we begin by constructing Haar wavelets in one spatial dimension and subsequently discuss how the process can be easily extended to piecewise constant functions over triangulations in two dimensions. The Haar scaling function is defined as follows:  

\begin{align*}
\phi(x) = \left \{ 
\begin{array}{lll}
1 & \text{if } &   \text{\hspace*{-.5 in}} x   \in [0,1)\\
0 & \text{otherwise.}
\end{array}
\right.
\end{align*}
The dilates and translates $\phi_{j,k}$  of $\phi$ are defined over $\mathbb{R}$ as
\begin{align*}
\phi_{j,k}(x) = 2^{j/2}\phi(2^{j}x-k)= 2^{j/2} 1_{\Delta_{j,k}}(x)
\end{align*}
for $j=0, \cdots, \infty$ and $k \in \mathbb{Z}$.
It is important to note that with this normalization the functions $\{\phi_{j,k}\}_{k=0}^{2^j-1}$ are $L^2[0,1]$ orthonormal so that
\begin{align*}
\langle \phi_{j,k},\phi_{j,l} \rangle = \int_{\mathbb{R}} \phi_{j,k}(x)\phi_{j,l}(x) \mathrm{d}x = \delta_{kl}.
\end{align*}
In this equation $\Delta_{j,k}=\{x|x \in [2^{-j}k,2^{-j}(k+1)]\}$ and $1_{\Delta_{j,k}}$ is the characteristic function of $\Delta_{j,k}$.
For any fixed integer $j$, the $\phi_{j,k}$ form a orthonormal basis that spans the space of piecewise constants over $\{\Delta_{j,k}\}_{k=0}^{2^j-1}$. We let $V_j$  denote space of piecewise constant functions
\begin{align}
V_j = \underset{k=0,\cdots, 2^{j-1}}{\mathrm{span}}\{ \phi_{j,k} \}.
\end{align}
Corresponding to Haar scaling function, the Haar wavelet $\psi$ is defined as
\begin{align*}
\psi(x)=\left \{
\begin{array}{lll}
1 & x \in [0,\frac{1}{2}),\\
-1 & x \in [\frac{1}{2},1).
\end{array}
\right.
\end{align*}
Again, the translates and dilates of $\psi_{j,k}$ of $\psi$ are given by
\begin{align*}
\psi_{j,k}(x) = 2^{j/2}\psi(2^{j}x-k),
\end{align*}
and the complement spaces $W_j$ are defined by $W_j= \underset{k}{\mathrm{span}}\{\psi_{j,k}\}$.
It is straightforward to verify that the spaces $V_{j-1}$ and $W_{j-1}$  form an orthogonal direct sum of $V_j$. That is, we have
\begin{align*}
V_j = V_{j-1} \bigoplus W_{j-1}.
\end{align*}
This process is well-known and standard in the literature as a means of constructing multiscale bases for $L^2[0,1]$. We will follow an analogous strategy to construct multiscale bases over the triangular domain depicted in Figure \ref{fig:refine}. We first denote the characteristic functions over the triangular domain as shown in the Figure \ref{fig:refine} 
\begin{align*}
1_{\Delta_{s}}(x)=\left \{
\begin{array}{lll}
1 & x \in \Delta_s\\
0 & \text{otherwise.}
\end{array}
\right.
\end{align*}
We next consider the regular refinement shown in Figure \ref{fig:refine} where $\Delta_{i_1 i_2}$ is the $i_2$ child of $\Delta_{i_1}$. In general $\Delta_{i_1 i_2\hdots i_m i_{m+1}}$ is the $(m+1)^{st}$ child of $\Delta_{i_1 i_2 \hdots i_m}$. The multiscaling function $\phi_{j,k}$ is defined as
\begin{align*}
\phi_{j,k}(x) = \frac{1_{\Delta_{i_1 i_2 \hdots i_j}}(x)}{\sqrt[]{m(\Delta_{i_1 i_2 \hdots i_j})}}
\end{align*}
where $j$ refers to the level of refinement in the grid.  

\begin{figure}
\centering
\begin{subfigure}[b]{.2\textwidth}
\includegraphics[width=3cm]{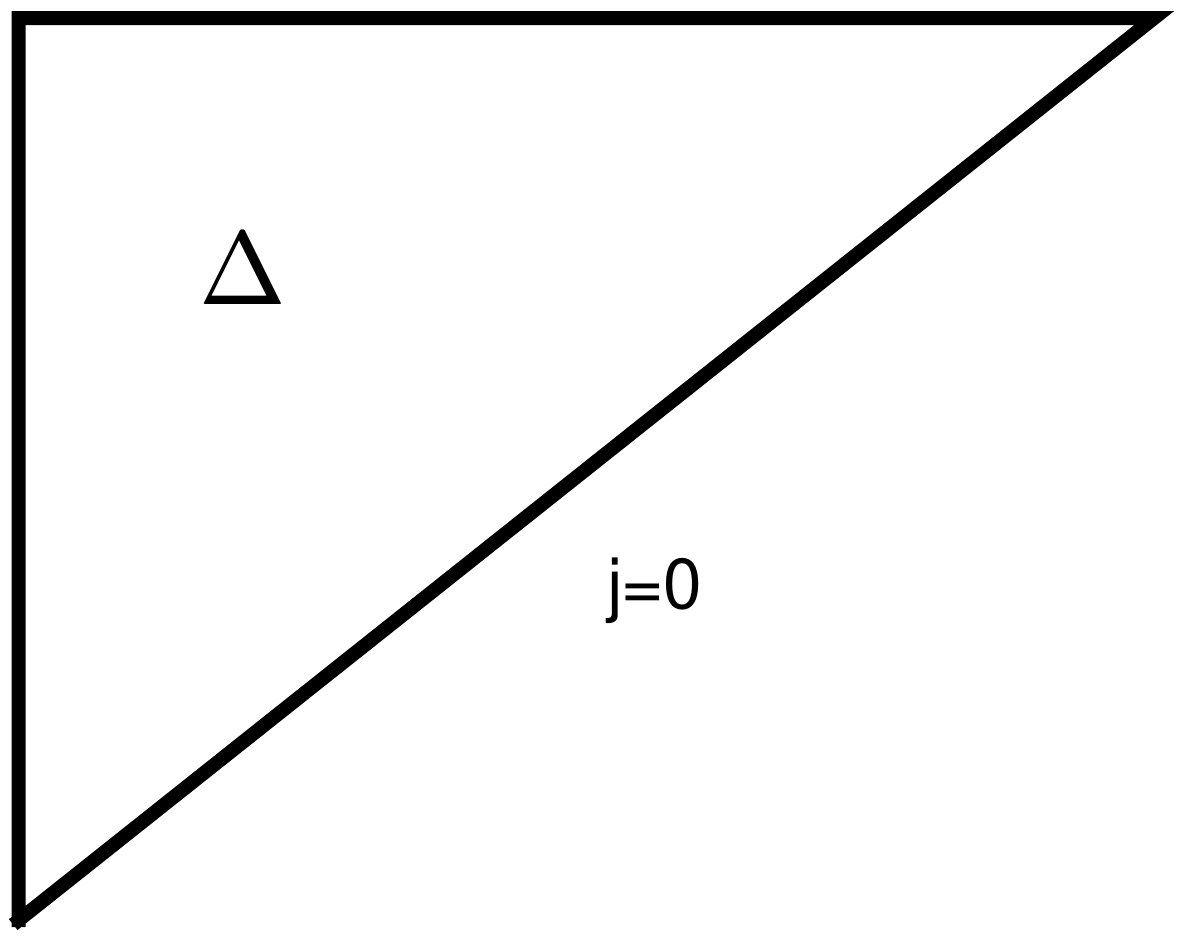}
\end{subfigure}
\begin{subfigure}[b]{.2\textwidth}
\includegraphics[width=3cm]{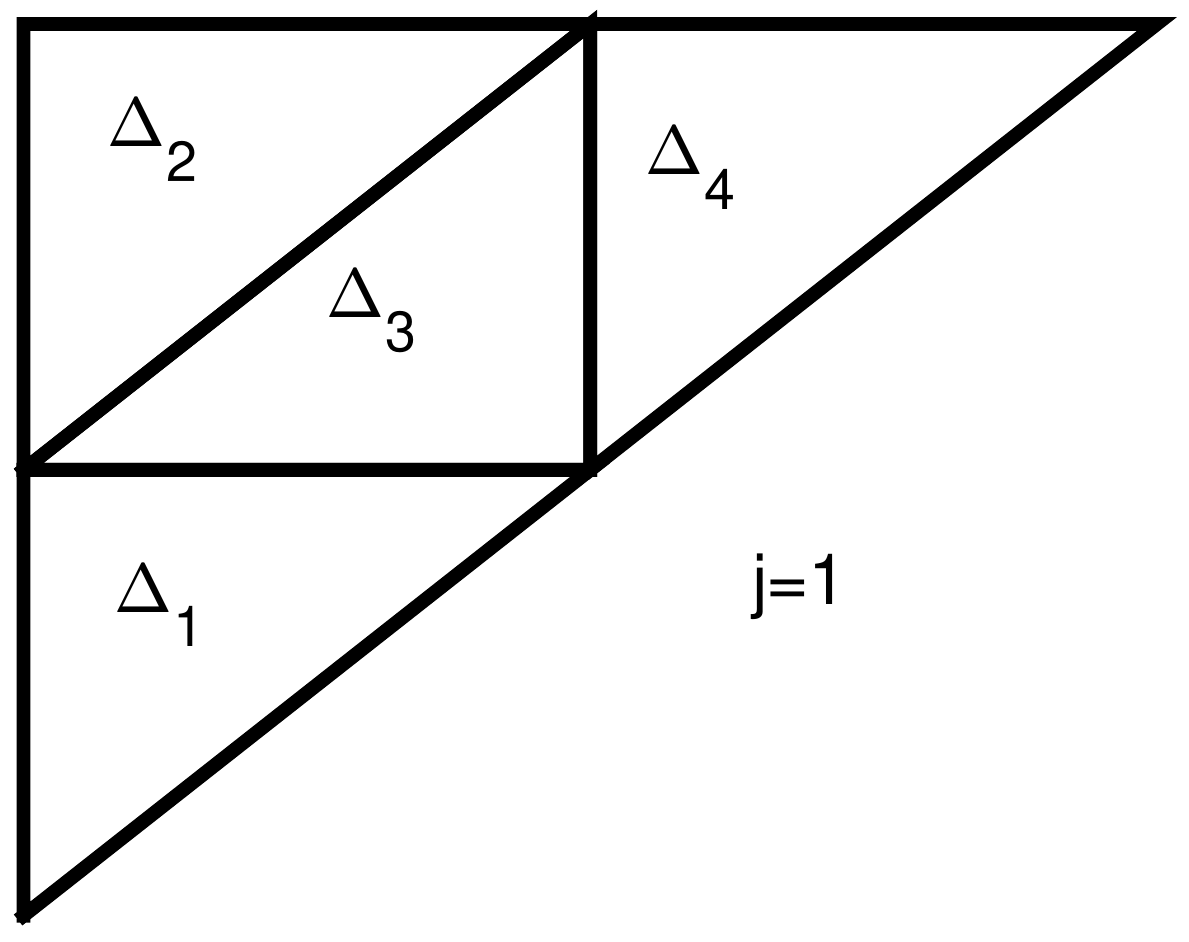}
\end{subfigure}
\begin{subfigure}[b]{.2\textwidth}
\includegraphics[width=3cm]{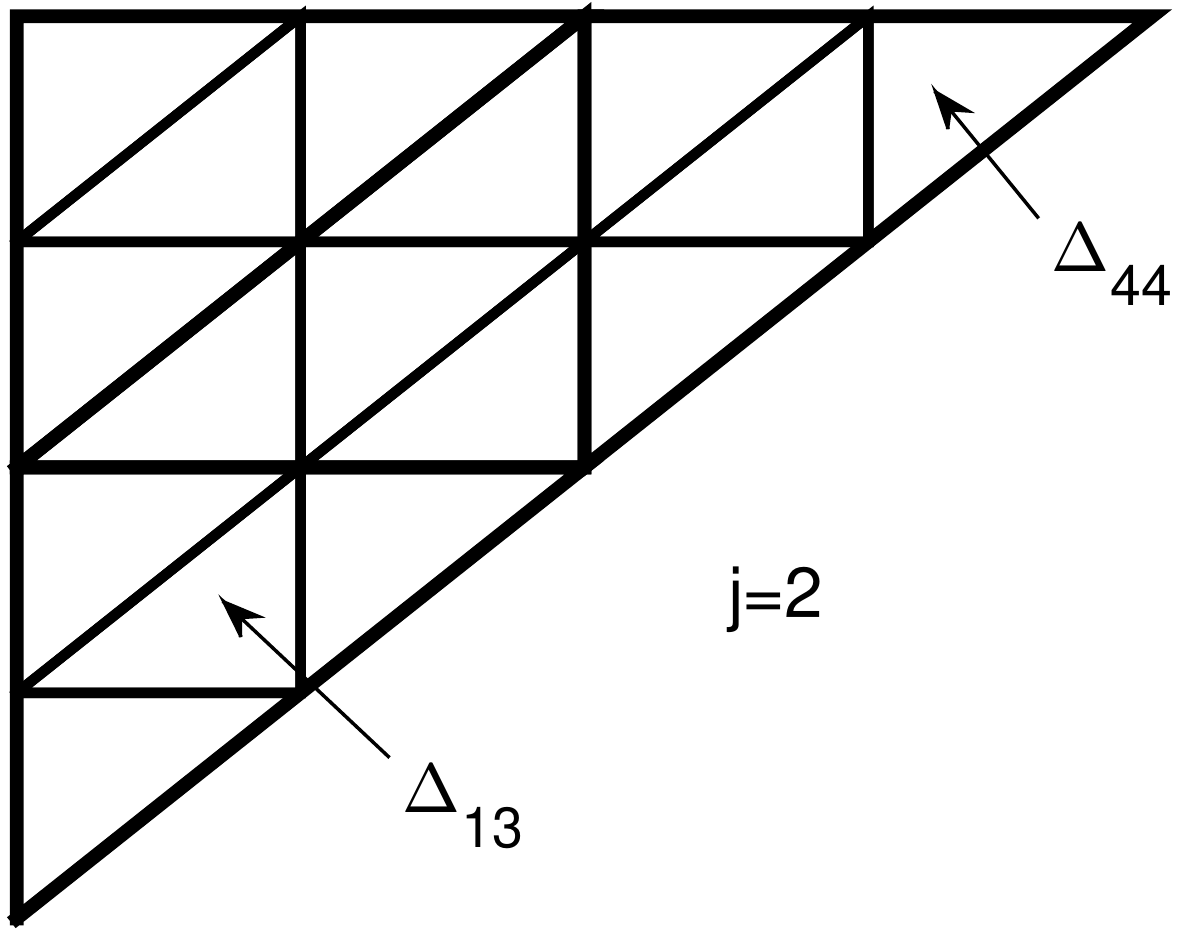}
\end{subfigure}
\begin{subfigure}[b]{.2\textwidth}
\includegraphics[width=3cm]{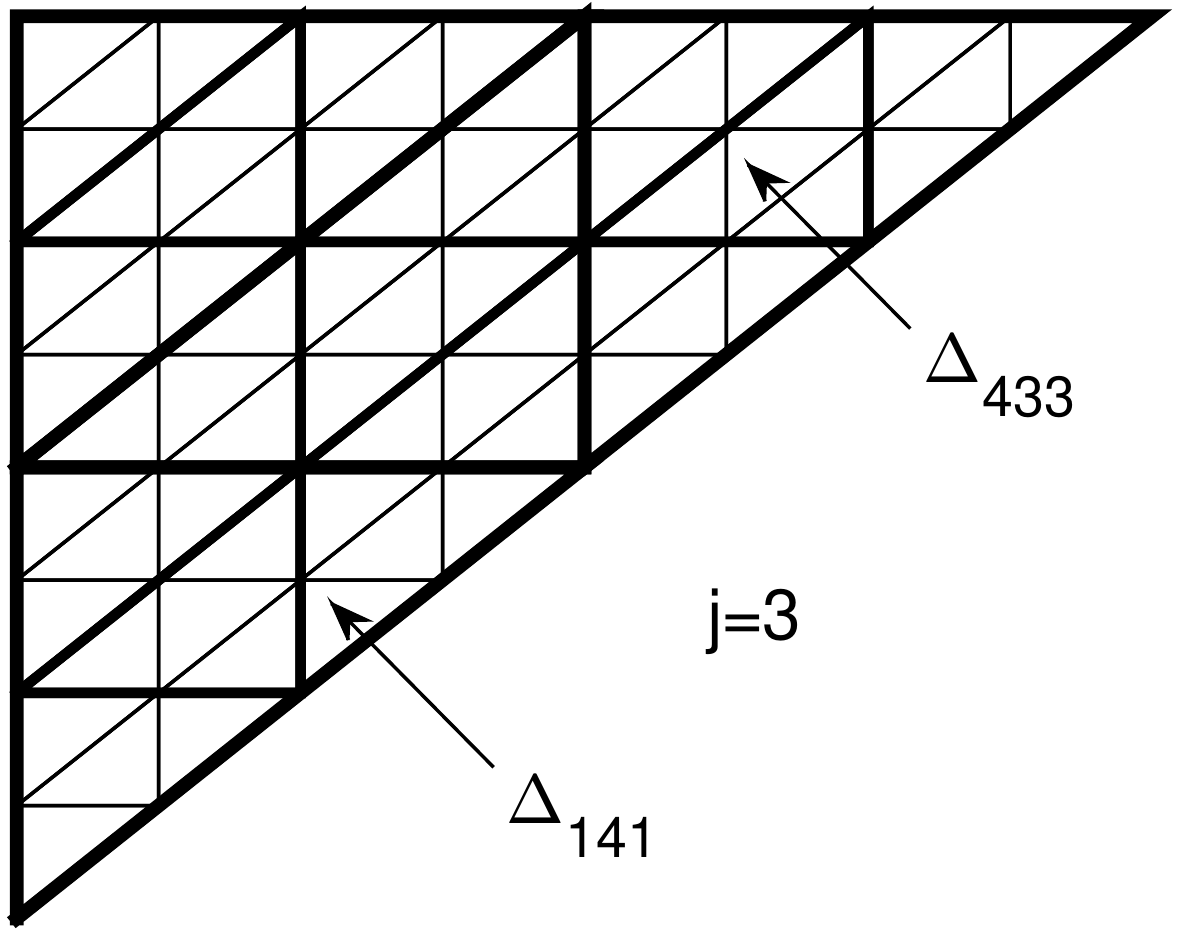}
\end{subfigure}
\caption{Regular refinement process for domain $\Delta$}
\label{fig:refine}
\end{figure}

Since the history dependent operators $(\mathcal{H}X)(t)$ act on the infinite dimensional space $P=P_1 \times \cdots \times P_\ell$ of functions $\mu=(\mu_1,\ldots, \mu_\ell)$, we need approximations of these operators for computations and applications.  In the discussion that follows we  choose each function $\mu_i\in P_i:=L^2(\Delta)$ where the domain $\Delta\subset \mathbb{R}^2$ is defined as  
$$
\Delta:=\left \{ (s_1,s_2) \in \mathbb{R}^2 \biggl | 
\underline{s}\leq s_1 \leq s_2 \leq \overline{s}\right \}.
$$
The modification of the construction that follows for different domains $\Delta_i$ for the functions $\mu_i \in L^2(\Delta_i)$ is trivial, but  notationally tedious, and we leave the more general case  to the reader. 
 Given the domain $\Delta$ we introduce a regular refinement depicted in Figure \ref{fig:refine} and disscused in more detail in Appendix A.  The set $\Delta$ is subdivided into $\Delta_1,\Delta_2,\Delta_3,\Delta_4$ as shown, and each $\Delta_i$ is subdivided into $\Delta_{i1},\Delta_{i2},\Delta_{i3},\Delta_{i4}$.  Further subdivision recursively introduces the sets  $\Delta_{{i_1}\ldots{i_j}}$  for $i_j=1,\ldots,4$ that are the children of $\Delta_{{i_1}\ldots{i_{j-1}}}$.  

The characteristic functions $1_{\Delta_1},\ldots, 1_{\Delta_4}$ define a collection of multiscaling functions $\phi^1,\ldots, \phi^4$ as defined in \cite{keinert2004}.  
We define the space of piecewise constant functions $V_j$ on grid refinement level $j$ to be the span of the characteristic functions
of the sets $\Delta_{i_1,\ldots,i_j}$, 
so that the dimension of $V_j$ is $4^j$. We denote by $\left \{ \phi_{j,k} \right \}_{k=1,\ldots, 4^j}$ the orthonormal basis obtained from these characteristic functions on a particular grid level, each normalized so that $(\phi_{j,k},\phi_{j,\ell})_{L^2(\Delta)} = \delta_{k,l}$. Each of the basis functions $\phi_{j,k}$ will be proportional  to  the characteristic function $\phi^\ell(2^j x + d)$ for some $\ell\in,{1,2,3,4}$ and displacement vector $d$. It is straightforward in this case \cite{keinert2004} to define  $3$ piecewise constant multiwavelets  $\psi^1, \psi^2, \psi^3$ that  that are used to define functions $\psi_{j,m}$ for $m=1,\ldots,3 \times 4^j$ that span  the complement spaces $W_j= \text{span}\left \{ \psi_{j,m}\right \}_{m=1,\ldots,3\cdot 4^j}$ that satisfy

\begin{align*}
\underbrace{V_j}_{4^j functions} = \underbrace{V_{j-1}}_{4^{j-1} functions}\bigoplus \underbrace{W_{j-1}}_{3 \times 4^{j-1} functions}.
\end{align*}
It is a straightforward exercise to define $L^2(\Delta)-$orthonormal wavelets that span $W_j$ for each $j \in \mathbb{N}_0$, but the nomenclature is lengthly. Since we do not use the wavelets specifically in this paper, the details are omitted. Each function $\psi_{j,m}$ is proportional to one of the three scaled and translated multiwavelet functions and   satisfies the orthonormality conditions
\begin{align*}
\begin{array}{lll}
(\psi_{j,k},\psi_{m,\ell}) = \delta_{j,m}   \delta_{k,\ell} &  & \text{for all } j,k,m,\ell,\\
(\psi_{j,k},\phi_{m,\ell}) = 0 & \quad \quad & \text{for $j\geq m$ and all $k,\ell$.} 
\end{array}
\end{align*}
In the next step, we denote the orthogonal projection onto the span of the piecewise constants defined on a grid of resolution level $j$ by $\Pi_j$ so that
$$\Pi_j:P \rightarrow V_j.$$
Finally, we define the approximation space
$\mathcal{A}^\alpha_2$ in terms of the projectors $\Pi_j$ as 
$$
\mathcal{A}^\alpha_2:= \left  \{
f\in P \biggl |
\|f\|_{\mathcal{A}^\alpha_2}:=
\left ( \sum_{j=0}^\infty 
2^{2\alpha j} \| (\Pi_j-\Pi_{j-1})f \|^2_P
\right )^{1/2}
\right \}.
$$
Note that this is a special case of the more general analysis in \cite{dahmen1996}.
We define our approximation method in terms of one point quadratures 
defined over the triangles $\Delta_{i_1\ldots i_j}$ that constitute the grid of level $j$ that defines $V_j$.  For notational convenience, we  collect all triangles at a fixed level $j$ in the singly indexed set
$$ 
\left \{ \Delta_{j,k} \right \}_{k\in \Lambda_j}:= \left \{ \Delta_{i_1\ldots i_j} \right \}_{i_1,\ldots,i_j \in {1,2,3,4}}
$$
\noindent where $\Lambda_j:=\left \{ k\in \mathbb{N} \quad | 1\leq k \leq 4^j \right \}$, and the quadrature points are chosen such that $\xi_{j,k}\in \Delta_{j,k}$ for $k = 1,\cdots, \Lambda_j$. 
We now can state our principle approximation result for the class of history dependent operators in this paper.
\begin{theorem}
Suppose that  the  function $\gamma$ that defines the history dependent kernel
in Equation \ref{eq:defgamma} is a bounded function in $C^{\alpha}(\mathbb{R})$, and define the approximation $h_j$ associated with the grid level $j$  of the history dependent operator $h$ to be 
\begin{equation*}
(h_jf)(t)\circ \mu:=  \iint_\Delta \left ( \sum_{\ell\in \Gamma_j} 
1_{\Delta_{j,\ell}}(s) \kappa(\xi_{j,\ell}, t,f) 
\right ) \mu(s) ds .
\end{equation*}
Then there is a constant $C>0$ such that 
\begin{equation}
\left |  
(h_jf)(t) \circ \mu - (hf)(t) \circ \mu 
\right | \leq C 2^{-(\alpha+1) j}
\label{eq:error_1}
\end{equation}
for all $f\in C[0,T]$, $t\in [0,T]$, and $\mu\in P$.  If in addition $\mu \in \mathcal{A}^{\alpha+1}_2$, there is a constant $\tilde{C}>0$ such that 
\begin{equation}
\left |  
(h_jf)(t) \circ \Pi_j \mu - (hf)(t) \circ \mu 
\right | \leq \tilde{C} 2^{-(\alpha+1) j}
\label{eq:error_2}
\end{equation}
for all $f\in C[0,T]$ and $t\in [0,T]$. 
\end{theorem}
\begin{proof}
We first prove the inequality in Equation \ref{eq:error_1}.  By definition of the operator $h$, we can write
\begin{align*}
\left |  (h_j f)(t) \circ \mu - (hf)(t) \circ \mu  \right | & \leq
\iint_\Delta \left | \left( 
\sum_{k\in \Lambda_j} 1_{\Delta_{j,k}}(s) \kappa(\xi_{j,k},t,f) 
-\kappa(s,t,f)\right)\mu(s)  \right | ds  \\
& \leq 
\iint_\Delta \left |  
\sum_{k\in \Lambda_j} 1_{\Delta_{j,k}}(s) \left (  \kappa(\xi_{j,k},t,f) 
-\kappa(s,t,f)  \right ) \right | \left |\mu(s)  \right | ds 
\end{align*}
Since the ridge function $\gamma$ is a bounded function in $C^{\alpha}(\mathbb{R})$, the output mapping $s \mapsto \kappa(s,t,f)$ is also a bounded function in $C^{\alpha}(\Delta)$ where the Lipschitz constant is independent of  $t\in [0,T]$ and $f\in C[0,T]$. Using Proposition 2.5  of \cite{visintin1994}, we have 
\begin{align*}
\left |  (h_j f)(t) \circ \mu - (hf)(t) \circ \mu  \right | & \leq
\iint_\Delta \left |  
\sum_{k\in \Lambda_j} 1_{\Delta_{j,k}}(s)  L \| \xi_{j,k} - s \|^\alpha   \right | \left |\mu(s)  \right | ds  \\
& \leq L \sum_{k\in \Lambda_j} \left ( 
m(\Delta_{j,k}) \left ( \frac{\sqrt{2} (\overline{s}-\underline{s})}{2^j}\right )^\alpha 
\right ) \iint_{\Delta_{j,k}} |\mu(s)|ds \\
& \leq L \sum_{k\in \Lambda_j} \left ( 
m(\Delta_{j,k}) \left ( \frac{\sqrt{2} (\overline{s}-\underline{s})}{2^j}\right )^\alpha 
\right ) m^{1/2}(\Delta_{j,k}) \| \mu \|_P \\
&\leq L 2^{2j} \left ( \frac{1}{2} \frac{(\overline{s}-\underline{s})^2}{2^{2j}} 
\left ( \frac{\sqrt{2}(\overline{s}-\underline{s})}{ 2^j} \right )^\alpha 
\right ) \left(\frac{1}{2} \frac{(\overline{s}-\underline{s})^2}{2^{2j}}\right)^{1/2} \| \mu \|_P = C 2^{-(\alpha+1) j} \| \mu \|_P
\end{align*}
Since we have 
\begin{equation*}
\left |  (h_j f)(t) \circ \Pi_j \mu - (hf)(t) \circ \mu  \right |  \leq
\left |  (h_j f)(t) \circ \Pi_j \mu - (h_jf)(t) \circ \mu  \right | 
+ 
\left |  (h_j f)(t) \circ  \mu - (hf)(t) \circ \mu  \right | , 
\end{equation*}
the second inequality in Equation \ref{eq:error_2} follows from the first Equation \ref{eq:error_1} provided we can show that 
\begin{equation*}
\left |  ( h_jf)(t) \circ \left (\mu - \Pi_j \mu \right )   \right | \leq C 2^{-(\alpha+1) j}
\end{equation*}
for some constant $C$. But it is a standard feature of the approximation spaces that if $\mu \in \mathcal{A}^{\alpha+1}_2$, then  $\|\mu-\Pi_j \mu \|_P \leq  2^{-(\alpha+1) j} \|\mu\|_{\mathcal{A}^{\alpha+1}_{2}}$. To see why this is so, suppose that $\mu\in \mathcal{A}^{\alpha+1}_2$. We have
\begin{align*}
\|\mu-\Pi_j \mu\|^2_P & = \sum_{k=j+1}^\infty \| (\Pi_k -\Pi_{k-1})\mu \|^2_P \\
 & \leq  \sum_{k=j+1}^\infty 2^{-2(\alpha+1) k} 2^{2(\alpha+1) k} 
 \| (\Pi_k -\Pi_{k-1})\mu \|^2_P \\
 &  \leq 2^{-2(\alpha+1) j} \sum_{k=j+1}^\infty 2^{2(\alpha+1) k} \| (\Pi_k -\Pi_{k-1})\mu \|^2_P  \leq 2^{-2(\alpha+1) j} \| \mu \|^2_{\mathcal{A}^{\alpha+1}_{2}}.
\end{align*}
When we apply this to our problem,  the upper bound follows immediately 
\begin{align*}
\left |  ( h_jf)(t) \circ \left (\mu - \Pi_j \mu \right )   \right | &\leq 
\sup_{t\in[0,T]}\left \| (h_jf)(t)\right \|_{P^*} \| \mu - \Pi_j \mu\|_P \\
 & \leq C2^{-(\alpha+1) j},
\end{align*} 
since the boundedness of the ridge function $\gamma$ implies the uniform boundedness  of the history dependent operators $(h_j f)(t)$ over $[0,T]$. 
\end{proof}
Theorem 1 can now be used to establish error bounds for input-output maps that have the form in Equation \ref{eq:hdef2}. 
\begin{theorem}
Suppose that the hypotheses of Theorem 1 hold. Then we have 
$$ \|(\mathcal{H}X)(t)-(\mathcal{H}_j X)(t)\Pi_j \| \lesssim 2^{-(\alpha+1)j}.$$
\end{theorem}
Where $\mathcal{H}$ is defined in Equations \ref{eq:2.1},\ref{eq:hdef2}, and $\mathcal{H}_j$ is defined in Equations \ref{eq:Hj1}, \ref{eq:Hj2} and \ref{eq:hj3} below.
\begin{proof}
Recall that for $i = 1 \dots q$ we had $$y_i(t) = \sum_{\substack{\ell = 1 \dots l}} b_{i\ell}(X(t))(h_\ell(a(X))(t)\circ \mu_\ell.$$
In matrix form this equation can be expressed as 
\begin{align*}
\begin{bmatrix}
y_1(t)\\
\vdots \\
y_q (t)
\end{bmatrix}
= 
\underbrace{
\begin{bmatrix}
b_{11}(X(t)) & \hdots &b_{1\ell}(X(t))\\
\vdots & \ddots & \vdots \\
b_{q1}(X(t)) & \hdots & b_{q\ell}(X(t))
\end{bmatrix}
}_{\text{$\mathbb{R}^{q \times l}$}}
\underbrace{
\begin{bmatrix}
h_1(a(X))(t)\circ \mu_1\\
\vdots \\
h_\ell(a(X))(t)\circ \mu_\ell
\end{bmatrix}
}_{\text{$(HX)(t)\circ\mu$}}.
\end{align*}
It follows that,
\begin{align*}
y(t) = \underbrace{(\mathcal{H}X)(t)\circ\mu}_{\text{$\in$ $\mathbb{R}^q$}} = \underbrace{ b(X(t))}_{\in \mathbb{R}^{q\times\ell}} \underbrace{(HX)(t) \circ \mu}_{\mathbb{R}^{\ell}}.
\end{align*}
By assumption $X \in C([0,T],\mathbb{R}^m)$. The construction of $H$ and $\mathcal{H}$ guarantees that 
$
(HX)(t) \in \mathcal{L}(P,\mathbb{R}^l), 
$

$
H : C\left([0,T],\mathbb{R}^m\right) \to C\left([0,T],\mathcal{L}(P,\mathbb{R}^l)\right),
$
and
$
\mathcal{H} : C([0,T],\mathbb{R}^m) \to C([0,T],\mathbb{R}^q).
$
 In this proof we denote by $(\mathbb{R}^l,\|.\|_u)$ the norm vector space that endows $\mathbb{R}^l$ with the $l^m$ norm $\|v\|_u:= \left( \sum^l_{i=1}|v_i|^u\right)^{\frac{1}{u}}$ for $1\leq u \leq \infty$. The normed vector space $(\mathbb{R}^{q\times l},\|.\|_{s,u})$ denotes the induced operator norm on matrices that map $(\mathbb{R}^l,\|.\|_u)$ into $(\mathbb{R}^q,\|.\|_s)$. Now we define an approximation on the mesh level $j$ of $\mathcal{H}$ to be \begin{equation}
 (\mathcal{H}_j X)(t) = b(X(t)((H_j X) (t)\Pi_j
 \label{eq:Hj1},
 \end{equation}
\begin{equation}
(H_j X)(t) = \begin{bmatrix}
h_{1,j}(a(X))(t) & 0 & \cdots & 0  \\
0 & h_{2,j}(a(X))(t) & 0 & \vdots \\
\vdots & \cdots & \ddots & 0 \\
0 & \cdots & 0 & h_{l,j}(a(X))(t)
\end{bmatrix}
\label{eq:Hj2}
\end{equation} and
\begin{equation}
h_{i,j}(t) \circ \nu = \iint_{\Delta} \sum_{k \in \Lambda_j}  1_{\Delta_{k,j}}(s) \kappa(\xi_{j,k},t,f) \nu (s) \mathrm{d}s
\label{eq:hj3}
\end{equation}
for $i = 1\dots \ell$ and $\nu_i\in P_i$.
To simplify the derivation or an error bound for approximation of $\mathcal{H}X(t)\circ \mu$, let $(HX)(t) \circ \mu $ be denoted by $g(t)$. We have assumed that $X \to b(X)$ and $t \to X(t)$ are continuous mappings. There fore $t \mapsto b(X(t))$ is continuous and on a compact set $[0,T]$, and $b(X(\cdot))$ $\in$ $C([0,T],\mathbb{R}^{q\times l})$. We therefore by definition have
\begin{align*}
\|b(X(t))\|_{(\mathbb{R}^{q\times l},\|\cdot\|_{s,u})} & \leq \sup_{\tau \in [0,T]}\|b(X(\tau))\|_{(\mathbb{R}^{q\times l},\|\cdot\|_{s,u})}, \\
& = \|b(X(\cdot))\|_{C([0,T],(\mathbb{R}^{q\times l},\|\cdot\|_{s,u}))}, \\ 
\|b(X(t))g(t)\|_{(\mathbb{R}^q,\|\cdot\|_q)} &\leq \|b(X(\cdot))\|_{C([0,T],(\mathbb{R}^{q\times l},\|\cdot\|_{s,u}))}\|g(t)\|_{(\mathbb{R}^l,\|\cdot\|_u)}. 
\end{align*}
with the norms explicitly denoted in the subscript. For $t\in[0,T]$, and applying these definitions,
{\small
\begin{align*}
\|b(X(t))\left((HX)(t)- (H_j X)(t)\Pi_j\right)\mu\|_{(\mathbb{R}^q,\|\cdot\|_s)} &\leq \|b(X(t))\|_{(\mathbb{R}^{q\times l},\|\cdot\|_{s,u})}\|((HX)(t)-(H_j X)(t)\Pi_j) \circ \mu \|_{(\mathbb{R}^l,\|\cdot\|_u)} \\
&\leq \|b(X(\cdot))\|_{(C([0,T],(\mathbb{R}^{q\times l},\|\cdot\|_{s,u}))}\|((HX)(t)-(H_j X)(t)\Pi_j) \circ \mu \|_{(\mathbb{R}^l,\|\cdot\|_u)}
\end{align*}
}
with
\begin{align*}
\|((&HX)(t)-(H_j X)(t)\Pi_j)  \circ \mu \|_{(\mathbb{R}^l,\|\cdot\|_u)} = \\
& \left \|\begin{bsmallmatrix}
h_{1,j}(a(X))(t) - h_{1,j}(a(X))(t)\Pi_j&&&&0  \\
 & h_{2,j}(a(X))(t) - h_{2,j}(a(X))(t)\Pi_j &&&\\
&&\ddots  &&\\
0&&&& h_{l,j}(a(X))(t)-h_{l,j}(a(X))(t)\Pi_j
\end{bsmallmatrix}
\begin{bsmallmatrix}
\mu_1\\
\mu_2\\
\vdots\\
\mu_l
\end{bsmallmatrix}\right \|_{(\mathbb{R}^l,\|\cdot\|_u)}.
\end{align*}
Therefore we can now write
\begin{align*}
\left\|((HX)(t)-(H_j X)(t)\Pi_j) \circ \mu \right\|_{\left(\mathbb{R}^l ,\|\cdot\|_u\right)} \leq \|((HX)(t)-(H_j X)(t)\Pi_j)\|_{(\mathcal{L}(P,(\mathbb{R}^l,\|\cdot\|_u)))}\|\mu\|_{P}.
\end{align*} Hence, recalling Theorem 1 we can now derive the convergence rate 
\begin{align*}
\|((HX)(t)-(H_j X)(t)\Pi_j)\|_{\left(\mathcal{L}\left(P,\left(\mathbb{R}^l,\|\cdot\|_u\right)\right)\right)} &= \sup_{\|\mu\|<1} \|\left(\left(HX\right)(t)-\left(H_j X\right)(t)\Pi_j\right) \circ \mu \|_{(\mathbb{R}^l,\|\cdot\|_u)} \\
&\leq \sup_{\|\mu\|<1} \left| ((h_{i,j} (a(X)))(t) - (h_{i,j}(a(X)))(t)\Pi_j) \circ \mu
\right| \\
&\leq \hat{C}2^{-(\alpha+1)j}.
\end{align*}
Therefore we obtain the final bound
\begin{align}
\|(\mathcal{H}X)(t)-(\mathcal{H}_j X)(t)\Pi_j)\|_{(\mathbb{R}^l ,\|\cdot\|_u)}\lesssim 2^{-(\alpha+1)j},
\end{align}
for all $t\in [0,T]$.
\end{proof}
\section{Well-Posedness: Existence and Uniqueness}
\label{sec:exist}
The history dependent governing equations studied in this paper are a special case of the more general class of abstract Volterra equations or functional differential equations. A general treatise on abstract Volterra equations can be found in \cite{corduneaunu2008}, while various  generalizations of theory for the  existence and uniqueness of functional differential equations have been given in \cite{driver1962}, \cite{rudakov1978}, \cite{irs2002}.  
We have noted in Section \ref{sec:intro} that the general form of the governing equations we consider in this paper have the form 
\begin{equation}
\dot{X}(t)=AX(t)+B((\mathcal{H}X)(t) \circ {\mu}+u(t))
\label{eq:first_order}
\end{equation}
 where the state vector $X(t) \in \mathbb
{R}^{m}$, the control inputs
$u(t) \in \mathbb{R}^{q}$,
$A \in \mathbb{R}^{m{\times}m}=\mathbb{R}^{2n{\times}2n}$ is a Hurwitz matrix, and 
 $B \in \mathbb{R}^{m{\times}q}$  is the control  input matrix.  
We make the following assumptions about the history dependent operators $\mathcal{H}$:
\begin{enumerate}
 \item[H1)] $ \mathcal{H} : C([0,\infty),\mathbb{R}^m ) \mapsto  C([0,\infty),P^*)$
 \item[H2)] $\mathcal{H}$ is causal in the sense that for all 
 $ x,y  \in C([0,\infty);\mathbb{R}^m)$, 
 \[x(\cdot) \equiv {y} (\cdot) \; \text{on} \; [0,{\tau}] \implies (\mathcal{H}x)({t})=(\mathcal{H}y)({t})\: \quad \forall \: {t} \in [0,\tau]. \]
 \item[H3)] Define  the closed set consisting of all  continuous functions $f$ that remain within radius $r$ of the initial condition $X_0$ over the closed interval $[t,t+h]$, 
 $$\overline{\mathcal{B}}_{[t,t+h],r} (X_0):=\begin{Bmatrix} f\in C([0,h),\mathbb{R}^m) \biggl | f(0)=X_0 \text{ and } \|f(s) - X_0\|_{\mathbb{R}^m} \leq r \text{ for } s\in[t,t+h] \end{Bmatrix},$$ 
 for a fixed $X_0\in\mathbb{R}^m$. For each $t\ge 0$, we assume that there exist $h,r,L>0$ such that 
 \begin{equation}
 \| (\mathcal{H}X)(s)-(\mathcal{H}Y)(s) \|_{P^*}  \leq L \| X-Y \|_{[t,t+h]} \quad \quad s\in[t,t+h]
 \label{eq:local_lip}
 \end{equation}
 for all $X,Y \in \overline{\mathcal{B}}_{[t,t+h],r}(X_0)$. 
\end{enumerate}
Our first result guarantees the existence and uniqueness  of a local solution to Equation \ref{eq:1st_order}, and also describes an important case when such local solutions can be extended to $[0,\infty)$. This theorem can be proven via  the existence and uniqueness Theorem 2.3 in \cite{irs2002} for functional delay-differential equations. However, since we are not interested in delay differential equations in this paper, but rather on a highly structured class of integral hysteresis operators, the proof can be much simplified.  
\begin{theorem}
\label{th:exist_general}
Suppose that the history dependent operator $\mathcal{H}$ satisfies the hypotheses (H1),(H2),(H3).  Then there is a $\delta>0$ such that Equation \ref{eq:first_order} has a solution $X\in C([0,\delta), \mathbb{R}^m)$. Suppose the interval $[0,\delta)$ is extended to the maximal interval $[0,\omega)\subset [0,\delta)$ over which such a solution exists.  If the solution is bounded, then $[0,\omega)=[0,\infty)$. 
\end{theorem}
\begin{corollary}
\label{th:exist_specific}
Suppose that the history dependent operator $\mathcal{H}$ in Equation \ref{eq:first_order} is defined as in Equation \ref{eq:hdef1} and \ref{eq:hdef2} in terms of a globally Lipschitz, bounded continuous ridge function $\gamma: \mathbb{R}\rightarrow \mathbb{R}$ in Equation \ref{eq:defgamma}. Then Equation \ref{eq:first_order} has a unique solution $X \in C([0,\infty),\mathbb{R}^m)$ for each $\mu\in P$.
\end{corollary}
\begin{proof}
For completeness, we outline a simplified version the proof of Theorem \ref{th:exist_general} for our class of history and parameter dependent equations.  As a point of comparison, the reader is urged to compare the proof below to the conventional proof for systems of nonlinear ordinary differential equations, such as in \cite{khalil}. If we integrate the equations of motion in time, we can define an operator $T:C([0,h),\mathbb{R}^m) \rightarrow C([0,h), \mathbb{R}^m)$ from 
\begin{align*}
 X(t) &= X_{0} + \int_0^t AX(\tau) + B((\mathcal{H}X)(\tau) \circ \mu + u(\tau) ) \mathrm{d} \tau, \\
 X(t) &= (TX)(t), 
\end{align*}
for all $t\in[0,h]$. 
As introduced in hypothesis (H3), we select  $h,r>0$ and define 
\[ 
\overline{\mathcal{B}}_{[0,h],r}(X_0) := \Big\{ X \in C([0,h),\mathbb{R}^m) \biggl | X(0) = X_0 , \| X_0 - X \|_{[0,\delta]} \leq r  \Big\} .
\]
such that the local Lipschitz condition in Equation \ref{eq:local_lip} holds.  
Now we consider restricting the equation to a subinterval $[0,\delta]\subseteq [0,h]$, and investigate conditions on $T$ that enable the application of the contraction mapping theorem.  We first study  what conditions on $\delta>0$ are sufficient to guarantee that $T: \overline{\mathcal{B}}_{[0,\delta],r}(X_0) \rightarrow 
\overline{\mathcal{B}}_{[0,\delta],r}(X_0)$.  We have
{\small
\begin{align*}
\| TX (t) - X_0 \|_{\mathbb{R}^m} & \leq \int_0^t \| AX(s) + B((\mathcal{H}X(s) \circ \mu + u(s)) \|_{\mathbb{R}^m} \mathrm{d}s\\
 & \leq \int_0^t  \biggl ( \| A \| \|X(s) - X_0 \|_{\mathbb{R}^m} +  \underbrace{\|AX_0 \|_{\mathbb{R}^m}}_{\leq\|A\| \|X_0\| = \text{$M_A$}} \\ 
 & \qquad \qquad + \| B \| \underbrace{(\| (\mathcal{H}X)(s) -(\mathcal{H}X_0)(s) \|_{P^*}}_{\text{$ \leq L \|X-X_0\|_{[0,\delta]} $}} \underbrace{\| \mu \|_P}_{\text{$M_\mu$}} +  \underbrace{\| (\mathcal{H}X_0)(s) \|_{P^*} }_{\text{$\leq M_H= \|\mathcal{H}X_0\|_{C([0,\delta],P^*)}$ }}+ \underbrace{\| u \|_{C([0,h),\mathbb{R}^p)}}_{\text{$\leq M_u= \|u\|_{C([0,\delta],\mathbb{R}^q)}$}})  \biggr) \mathrm{d}s\\
 &\leq ((\|A\| + \| B \|M_\mu  L)r + M_A + \|B\|M_T)t\\
 & \leq ((\|A\| + \| B \|M_\mu  L)r + M_A + \|B\|M_T)\delta
\end{align*}
}
where $M_T=M_H+M_u$. 
Now we restrict $\delta$   so that 
 \[ 
 \Big((\|A\| + \| B \|M_\mu L)r + M_A + \|B\|M_T \Big)\delta \leq r, \]
 which implies
 \[ 
  \delta < \frac{r}{(\|A\| + \| B \|M_\mu L)r + M_A + \|B\|M_T)} .
 \]
 We thereby conclude that 
 \[
 \| TX (t) - X_0 \|_{C([0,h),\mathbb{R}^p)} \leq r \: \quad \text{for}  \quad \: t \in [0,\delta], \]
and it follows that $T : \overline{\mathcal{B}}_{[0,\delta],r} \rightarrow  \overline{\mathcal{B}}_{[0,\delta],r} $. 
 Next we study conditions on $\delta$ that guarantee that 
  $T : \overline{\mathcal{B}}_{[0,\delta],r} \rightarrow  \overline{\mathcal{B}}_{[0,\delta],r} $ is a contraction.  We  compute directly a bound on the difference of the output  as 
 \begin{align*}
  \| (TX) (t) - (TY)(t) \|_{\mathbb{R}^m} &\leq \int_0^t \| AX(s)-AY(s) + B((\mathcal{H}X)(s) - (\mathcal{H}Y(s))\circ \mu ) \|_{\mathbb{R}^m} \mathrm{d}s \\
  &\leq (\| A \|+\| B \| M_\mu L_\mu )\|X-Y\|_{\mathbb{R}^m}\delta . 
 \end{align*}
If we choose 
\[ \delta < min \biggl \{ h, \frac{r}{(\|A\| + \| B \|M_\mu L)r + M_A + \|B\|M_T)},\frac{1}{ \| A \|+\| B \| M_\mu L } \biggr \}, 
\]
it is apparent that $T$ is a contraction that maps  the closed set $\overline{\mathcal{B}}_{[0,\delta],r}$ into itself.   
There is a unique solution in $\overline{\mathcal{B}}_{[0,\delta],r}$ on $[0,\delta]$. 
\end{proof}
\section{Online Identification}
\label{sec:ident}
A substantial literature has emerged that treats online estimation problems for linear or nonlinear plants governed by systems of ordinary differential equations.  Approaches for these finite dimensional systems that are based on variants of  Lyapunov's direct method can be found in any of a number of good texts including, for instance, \cite{na2005}, \cite{sb2012}, or \cite{is2012}. The general strategies that have proven fruitful for such finite dimensional systems have often been extended to classes of systems whose dynamics evolve in an infinite dimensional space: distributed parameter systems.  A discussion of the general considerations for identification of distributed parameter systems can be found in \cite{bk1989}, for example, while  studies that are specifically relevant to this paper  include \cite{d1993}, \cite{dr1994}, \cite{dr1994pe}, and \cite{bsdr1997}.

In this section we adapt the  framework introduced in  \cite{bsdr1997} to our class of history dependent, functional differential equations.  The approach in \cite{bsdr1997} assumes that the state equations for the distributed parameter system have first order form, and they are cast in terms of a nonlinear, parametrically dependent bilinear form that is coercive.  The resulting equations that govern the error in state  and in distributed parameter estimates is a nonlinear function of the state trajectory of the plant. In contrast, a similar strategy in this paper yields error equations that depend nonlinearly on the history of the state trajectory.

The general online estimation problem discussed in this section assumes that we observe the value of the state $X(t)\in \mathbb{R}^m$ at each time $t\ge 0$ that depends on some unknown distributed parameter $\mu \in P$, and subsequently use the observed state to construct  estimates $\hat{X}$  of the states and $\hat{\mu}$ of the distributed parameters. We construct online estimates that evolve on the state space $\mathbb{R}^m \times P$ according to the time varying, distributed parameter system equations 
\begin{align}
 \dot{\hat{X}}(t) & = A \hat{X}(t) + B\left ( \left (  \mathcal{H} X \right )(t) \circ \hat{\mu}(t) + u(t)  \right ), \notag \\
 \dot{\hat{\mu}}(t) & = -\left (B (\mathcal{H}X)(t) \right)^*\hat{X}(t),
 \label{eq:infdim}
\end{align}
for $t\ge 0$ where the initial conditions are $\hat{X}_0:=X_0$, $\hat{\mu}(0):=\mu_0$. In these equations, we denote the adjoint operator $L^*$ for any bounded linear operator $L$. These equations can be understood as incorporating a natural choice of a parameter update law.  The learning law above can be interpreted as generalization of the conventional gradient update law that features prominently in  approaches for finite dimensional systems \cite{is2012} and that has been extended to distributed parameter systems in \cite{bsdr1997}. It is immediate that the 
error in estimation of the states $\tilde{X}:=X-\hat{X}$ and in the distributed parameters $\tilde{\mu}:=\mu-\hat{\mu}$ satisfy the homogeneous system of equations 
\begin{align*}
\begin{Bmatrix} \dot{\tilde{X}}(t) \\ \dot{\tilde{\mu}}(t) \end{Bmatrix}
=
\begin{bmatrix} A & B(\mathcal{H}X)(t) \\ -\left ( B(\mathcal{H}X)(t)\right)^*& 0  \end{bmatrix}
\begin{Bmatrix} {\tilde{X}}(t) \\ {\tilde{\mu}}(t)\end{Bmatrix}.
\end{align*}
\subsection{Approximation of the Estimation Equations}
\label{subsec:approx_est}
The governing system in Equations \ref{eq:infdim} constitute a distributed parameter system since the functions $\hat{\mu}(t)$ evolve in the infinite dimensional space $P$.  In practice these equations must be approximated by some finite dimensional system. We define $\tilde{X}_j=\hat{X}-\hat{X}_j$ and $\tilde{\mu}_j=\hat{\mu}-\hat{\mu}_j $ where $\tilde{X}_j$ and $\tilde{\mu}_j$ express approximation errors due to projection of solutions in $\mathbb{R}^m \times P $ to a finite dimensional approximation space. We construct a finite dimensional approximation of the the online estimation equations using the results of Section \ref{subsec:approx_hist} and obtain
\begin{align}
\dot{\hat{X}}_j(t)  & = A\hat{X}_j(t) + B \left (
(\mathcal{H}_j X)(t) \Pi_j \circ \hat{\mu}_j(t) + u(t) 
\right ), \\
\dot{\hat{\mu}}_j(t) & = - \left ( B(\mathcal{H}_{j} X)(t) \Pi_j  \right)^* X(t).
\label{eq:approx_on_est}
\end{align}
\begin{theorem}
Suppose that the history dependent operator $\mathcal{H}$ in Equation \ref{eq:first_order} is defined as in equation \ref{eq:hdef1} and \ref{eq:hdef2} in terms of a globally Lipschitz, bounded continuous ridge function $\gamma: \mathbb{R}\rightarrow \mathbb{R}$ in Equation \ref{eq:defgamma}. Then for any $T>0$, we have
\begin{align*}
\| \hat{X} - \hat{X}_j\|_{C([0,T],\mathbb{R}^m)} & \rightarrow 0, \\
\|\hat{\mu} - \hat{\mu}_j\|_{C([0,T],P)} & \rightarrow 0,
\end{align*}
as $j\rightarrow \infty$.
\end{theorem}
\begin{proof}
Define the operators $G(t):P\rightarrow \mathbb{R}^m$ and $G_j(t):P \rightarrow \mathbb{R}^m$ for each $t\geq 0$ as
\begin{align*}
G(t)  & := B (\mathcal{H}X)(t),  \\
G_j(t)  & := B (\mathcal{H}_jX)(t) \Pi_j.
\end{align*}
The time derivative of the error in approximation can be expanded as follows:
{\small 
\begin{align*}
\frac{1}{2} \frac{d}{dt}\left (
( {\tilde{X}}_j, {\tilde{X}}_j )_{\mathbb{R}^m} + ({\tilde{\mu}}_j,{\tilde{\mu}}_j   )_P 
\right ) & = 
( \dot{\tilde{X}}_j, {\tilde{X}}_j )_{\mathbb{R}^m} + (\dot{\tilde{\mu}}_j,{\tilde{\mu}}_j   )_P 
   \\
&= (A\tilde{X}_j + G\hat{\mu} - G_j \hat{\mu}_j , \tilde{X}_j)_{\mathbb{R}^m} + \left ( -(G-G_j)^* X, \tilde{\mu}_j\right )_P \\
&= (A\tilde{X}_j, \tilde{X}_j)_{\mathbb{R}^m} 
+ \left (  (G-G_j)\hat{\mu}, \tilde{X}_j \right )_{\mathbb{R}^m}
+ \left ( G_j(\hat{\mu}-\hat{\mu}_j), \tilde{X}_j  \right )_{\mathbb{R}^m}
- \left ( (G-G_j)\tilde{\mu}_j, X \right )_{\mathbb{R}^m} \\
&\leq c (\tilde{X}_j, \tilde{X}_j)_{\mathbb{R}^m } 
+ \|(G-G_j)\hat{\mu} \|_{\mathbb{R}^m} \| \tilde{X}_j \|_{\mathbb{R}^m} +\\ & \qquad \qquad \qquad \qquad \| G_j\|_{\mathcal{L}(P,\mathbb{R}^m)} \| \tilde{\mu}_j \|_{P} \| \tilde{X}_j \|_{\mathbb{R}^m}+ \\ & \qquad \qquad \qquad  \qquad  \| G-G_j\|_{\mathcal{L}(P,\mathbb{R}^m)} \|\tilde{\mu}_j \|_{P} \| X\|_{\mathbb{R}^m}.
\end{align*}
}

We will next use a common inequality that can be derived from two applications of the triangle inequality.  We have
\begin{align*}
(a+b,a+b)=(a,a) + 2(a,b) + (b,b) &\geq 0, \\ 
(a-b,a-b)=(a,a) - 2(a,b) + (b,b) & \geq 0. 
\end{align*}
We conclude from this pair of inequalities that 
$$
|(a,b)| \leq \frac{1}{2} \left ( \|a\|^2 + \|b\|^2  \right ).
$$
The specific form that we apply this theorem is written as
\begin{equation}
|(a,b)|= |(\sqrt{\epsilon} a, \frac{1}{\sqrt{\epsilon}} b)|
\leq \epsilon \frac{\|a\|^2}{2} + \frac{1}{\epsilon} \frac{\|b\|^2}{2}.
\label{eq:ip_ab}
\end{equation}
We apply the inequality in Equation \ref{eq:ip_ab} to each term in which $\tilde{\mu}_j$ and $\tilde{X}_j$ appear in a product.
\begin{align*}
\frac{1}{2} \frac{d}{dt} \left ( 
\| \tilde{X}_j \|^2_{\mathbb{R}^m} +  \| \tilde{\mu}_j \|^2_{P)}
 \right) & \leq 
c \| \tilde{X}_j \|_{\mathbb{R}^m}^2 + \frac{1}{2a} \| (G-G_j)\hat{\mu} \|^2_{\mathbb{R}^m} 
+ \frac{a}{2} \|\tilde{X}_j \|^2_{\mathbb{R}^m} \\
& \text{\hspace*{.2in}} 
+ \frac{1}{2b} \| G_j \tilde{\mu}_j\|^2_{\mathbb{R}^m} + \frac{b}{2} \|\tilde{X}_j \|^2_{\mathbb{R}^m}
+ \frac{1}{2c} \| \tilde{\mu}_j \|^2_{P} \\& \text{\hspace*{.2in}}  + \frac{c}{2} \| G-G_j\|^2_{\mathcal{L}(P,\mathbb{R}^m)} \|X\|^2_{\mathbb{R}^m} .
\end{align*}
%
Then
\begin{align*}
 \frac{d}{dt} \left ( 
\| \tilde{X}_j \|^2_{\mathbb{R}^m} +  \| \tilde{\mu}_j \|^2_{P} 
\right ) 
& \leq  c \| G-G_j \|^2_{\mathcal{L}(P,\mathbb{R}^m)} \|X\|^2_{\mathbb{R}^m} 
+ (2c + a + b)\|\tilde{X}_j\|^2_{\mathbb{R}^m} \\ & 
\text{\hspace*{.2in}}+ \left ( 
\frac{1}{c} + \frac{1}{b} \|G_j^*G_j \| 
\right ) \| \tilde{\mu}_j\|^2_{P}  + \frac{1}{a} \|G - G_j \|^2_{\mathcal{L}(P,\mathbb{R}^m)} \|\hat{\mu}\|^2_{P}. 
\end{align*}
We integrate this inequality in time from $0$ to $t$ to obtain
{\small
\begin{align*}
\| \tilde{X
}_j(t) \|^2_{\mathbb{R}^m} +   \| \tilde{\mu}_j(t) \|^2_{P} 
 & \leq 
\| \tilde{X}_j(0) \|^2_{\mathbb{R}^m} +  \| \tilde{\mu}_j(0) \|^2_{P} \\ & 
 \text{\hspace*{.2 in}} + \int_0^t c \|G(s) - G_j(s) \|^2_{\mathcal{L}(P,\mathbb{R}^m)} \| X(s)\|^2_{\mathbb{R}^m} ds \\
& \text{\hspace*{.2in}} + \int_0^t \left \{ 
(2c+a+b) \|\tilde{X}(s)\|^2_{\mathbb{R}^m} 
+ \left ( \frac{1}{c} + \frac{1}{b} \| G^*_j(s) G_j(s) \|  \right ) \|\tilde{\mu}_j\|^2_{P} \right \}ds \\  & \text{\hspace*{.2 in}}+ \int_0^t  \frac{1}{a} \|G(s) - G_j(s) \|^2_{\mathcal{L}(P,\mathbb{R}^m)} \|\hat{\mu}\|^2_{P} 
 ds.
\end{align*}
}
Choose $a,b>0$ large enough so that $(2c+a+b)>0$ and set $\gamma>1$.  If we define 
\begin{align*}
\gamma &:=\max \left (
2c+a+b, \frac{1}{c}+ \frac{\eta}{b}  \sup_{s\in [0,T]} \|G^*(s)G(s)\|,1
\right ), \\
\lambda_j(t)&:= \|(I-\Pi_j)\hat{\mu}(0)\|_{P} + \int_0^t  \|G(s) - G_j(s)\|^2_{\mathcal{L}(P,\mathbb{R}^m)} \left(c \|X\|^2_{\mathbb{R}^m}+ \frac{1}{a} \| \hat{\mu}\|^2_{P} \right) ds,
\end{align*}
then the inequality can be written as 
\begin{align*}
\|\tilde{X}_j(t)\|^2_{\mathbb{R}^m} + \|\tilde{\mu}_j(t)\|^2_{P} \leq \lambda_j(t)
+ \gamma \int_0^t \left (  \|\tilde{X}_j(s)\|^2_{\mathbb{R}^m} + \| \tilde{\mu}_j(s) \|^2_{P} \right ) ds.
\end{align*}
Gronwall's Inequality now completes the proof of the theorem (see Appendix C). 
\end{proof}

We also further investigate $\lambda_j(t)$ to derive the convergence rate for the approximate states and parameters evolving associated with level $j$ resolution. According to the convergence results  obtained in Theorem 1 we have $\|G(s)-G_j(s)\|_{\mathcal{L}(P,\mathbb{R}^m)} = \|B(\mathcal{H}X)(t)-B(\mathcal{H}_j X)(t)\Pi_j\| 
\leq C_2 2^{-(\alpha+1)j}$. Therefore, $\|G(s)-G_j(s)\|^2 \leq C_2^{2} 2^{-(\alpha+1)2j}$. It then follows that
\begin{align*}
\lambda_j(t) &= \|(I-\Pi_j)\hat{\mu}(0)\|_{P}+ \int_0^t 2^{-(\alpha+1)2j} \left(c \|X\|^2_{\mathbb{R}^m}+ \frac{1}{a} \| \hat{\mu}\|^2_{P} \right)\mathrm{d}s \\
&\leq \|(I-\Pi_j)\| \| \hat{\mu}(0)\|_{P}+ 2^{-(\alpha+1)2j}\left(c \|X\|^2_{\mathbb{R}^m}+ \frac{1}{a} \| \hat{\mu}\|^2_{P} \right) t.\\
\end{align*}
If $t$ $\simeq$ $C_3 2^{(\alpha +1)j}$, then
$
\lambda_j(t) < \mathcal{O}(2^{-(\alpha+1)j}) \quad \text{for} \quad t \in [0,C_3 2^{(\alpha +1)j}].
$

\section{Adaptive Control Synthesis}
\label{sec:ident}
In order to estimate the function $\mu$ that weighs the contribution of history dependent kernels to the equations of motion, we first map it to an n-dimensional subspace of square integrable functions using a projection operator $\Pi^n: P \mapsto P^n$.
Let  
\begin{equation}
\dot{X}=AX+B((\mathcal{H}X) \circ ( \mu- \hat{\mu})+v)
\end{equation}
be the governing equation of a robotic system after applying a feedback linearization control signal as mentioned in Equation \ref{eqn:3} with $u=v-(\mathcal{H}X) \circ  \hat{\mu}$. We substitute $\mu = \Pi^n \mu+ (I-\Pi^n) \mu$ and write
\begin{equation}
\dot{X}=AX+B((\mathcal{H}X) \circ  ( \Pi^n \mu- \hat{\mu})+v)+B((\mathcal{H}X) \circ (I- \Pi^n) \mu).
\end{equation}
Finally, by replacing $d= \{(\mathcal{H}X)(I-\Pi^n) \circ \mu \}$ we obtain
\begin{equation}
\dot{X}=AX+B((\mathcal{H}X) \circ ( \Pi^n  \tilde{\mu})+v+d),
\end{equation}
where
\begin{equation}
\dot{\tilde{\mu}} = -((\mathcal{H}X) \Pi ^n )^* B^T P X.
\end{equation}
\begin{theorem} Suppose the state equations have the form of Equation \ref{eqn:3} and the matrix $\mathcal{P}$ is a symmetric positive definite solution of the Lyapunov equation $A^T\mathcal{P}+\mathcal{P}A=-Q$ where $Q>0$. Then by employing the update law $\dot{\tilde{\mu}} = -((\mathcal{H}X) \Pi ^n )^* B^T \mathcal{P} X$, the control signal
 \begin{align}
    v(t)= 
\begin{cases}
    -k\frac{B^T \mathcal{P} X}{\|B^T \mathcal{P} X\|},& \text{ if } \|B^T \mathcal{P} X\|\geq \epsilon\\
    -\frac{k}{\epsilon}B^T \mathcal{P} X,              & \text{ if } \|B^T \mathcal{P} X\|<\epsilon
\end{cases}
\label{eqn:sliding_C}
\end{align}
with $k>\|d\|$ drives the tracking error dynamics of the closed loop system is uniformly ultimately bounded and its norm is eventually  $O(\epsilon)$.

\end{theorem}

\begin{proof}
We choose the Lyapunov function 
\begin{equation}
V=\frac{1}{2}X^T \mathcal{P} X + \frac{1}{2}\left( \tilde{\mu},\tilde{\mu}\right)_{P}
\end{equation} 
where $\mathcal{P}$ is the solution of the Lyapunov equation $ A^T \mathcal{P} + \mathcal{P} A = -Q$. The derivative of the Lyapunov function $V$ along the closed loop system trajectory is
\begin{align*}
\dot{V} &=\frac{1}{2} (\dot{X}^T \mathcal{P} X + X^T \mathcal{P} \dot{X})+\left( \dot{\tilde{\mu}},\tilde{\mu}\right)_p\\
&= \frac{1}{2}\big(AX+B((\mathcal{H}X) \circ ( \Pi^n  \tilde{\mu})+v+d)\big)^T \mathcal{P} X +X^T \mathcal{P} (AX+B((\mathcal{H}X) \circ ( \Pi^n  \tilde{\mu})+v+d)\big) + \left( \dot{\tilde{\mu}},{\mu}\right)_P \\
&= \frac{1}{2}X^T (A^T \mathcal{P} + \mathcal{P} A)X+ X^T PB ( v +d) +X^T \mathcal{P} B \big((\mathcal{H}X) \circ ( \Pi^n  \tilde{\mu})\big) + \left( \dot{\tilde{\mu}},{\mu}\right)_P\\
&= -\frac{1}{2}X^T Q X + X^T \mathcal{P} B (v +d) + \left( \dot{\tilde{\mu}} + ((\mathcal{H}X) \Pi ^n )^* B^T \mathcal{P} X, \tilde{\mu}\right)_P\\
&=  -\frac{1}{2}X^T Q X + X^T  \mathcal{P} B ( v +d).
\end{align*}
Therefore we have 
\begin{align*}
\dot{V} & \leq   -\frac{1}{2}X^T Q X + X^T  \mathcal{P} B ( v +d), \\
& \leq -\frac{1}{2}X^T Q X + \begin{cases} 
X^T\mathcal{P}B\left(-k\frac{B^T \mathcal{P} X}{\|B^T \mathcal{P} X\|} +d \right) & \text{ if } \|B^T \mathcal{P} X\|\geq \epsilon \\
X^T\mathcal{P}B\left( -\frac{k}{\epsilon}B^T \mathcal{P} X +d \right) & \text{ if } \|B^T \mathcal{P} X\|\leq \epsilon
\end{cases}, \\ & \leq
 -\frac{1}{2}X^T Q X +  \begin{cases}
-\left(k-\|d\|\right)\|B^T\mathcal{P}X\|& \text{ if } \|B^T \mathcal{P} X\|\geq \epsilon\\ 
\epsilon k & \text{ if } \|B^T \mathcal{P} X\|\leq \epsilon
\end{cases}\\
& \leq -\frac{1}{2}X^T Q X+ \epsilon k.
\end{align*}
By Theorem 4.18 in \cite{khalil} we conclude that there is a $\bar{T}>0$ and $\tau>0$ such that $\|X(t) \| \leq \bar{C} \epsilon $ for all $t\geq \bar{T}$.


\end{proof}

\section{Numerical Simulations}
Our principle approximation result, the proposed online identification, and adaptive control of systems with history dependent forces are verified in this section.  In the first experiment, we validate the operator approximation error bound presented in Theorem 1. In the second experiment, we model a wind tunnel single wing section with a leading and trailing edge flaps and apply the proposed sliding mode adaptive controller presented in Theorem 4. We illustrate the stability of the closed loop system and convergence of the closed-loop system trajectories to the equilibrium point.
\label{sec:numerical}
\subsection{Operator Approximation Error}
In this section we consider a collection of numerical experiments to validate the operator approximation rates derived in Theorem 1.
In order to show that Equation \ref{eq:error_2} holds, we choose a function $\mu(s)$ over $\Delta$ and then calculate $(h_j f)(t)\circ \mu_j$ for different levels of refinement.
Since the computation of $(hf)(t)\circ \mu$ exactly is numerically infeasible, we choose $J \gg j$ as the finest level of refinement in our simulation. According to Theorem 1, we have
\begin{equation*}
|(h_J f)(t)\circ \mu_J-  (hf)(t)\circ \mu|
 \leq C_J 2^{-(\alpha+1)J},
\end{equation*}
and for $j\ll J$ we see that
\begin{equation*}
|(h_j f)(t)\circ \mu_j-  (hf)(t)\circ \mu|
\leq C_j 2^{-(\alpha+1)j}.
\end{equation*}
Assuming $C=\max\{C_j,C_J\}$ and using the triangle inequality, we obtain
\begin{equation}
\begin{aligned}
|(h_J f)(t)&\circ \mu_J-(h_j f)(t)\circ \mu_j|\\ &\leq C(2^{-(\alpha+1)J}+2^{-(\alpha+1)j})
\end{aligned}
\label{eq:numerical}
\end{equation}
Therefore, given the weights $\mu_J$ for the finest level of refinement $J$, we can evaluate $\mu_j=\Pi_j \mu_J$ and numerically verify Equation \ref{eq:numerical}.
\begin{figure}[h!]
\centering
\begin{subfigure}[b]{.03\textwidth}
\includegraphics[width=.5cm]{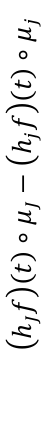}
\end{subfigure}
\begin{subfigure}[b]{.65\textwidth}
\includegraphics[width=9cm]{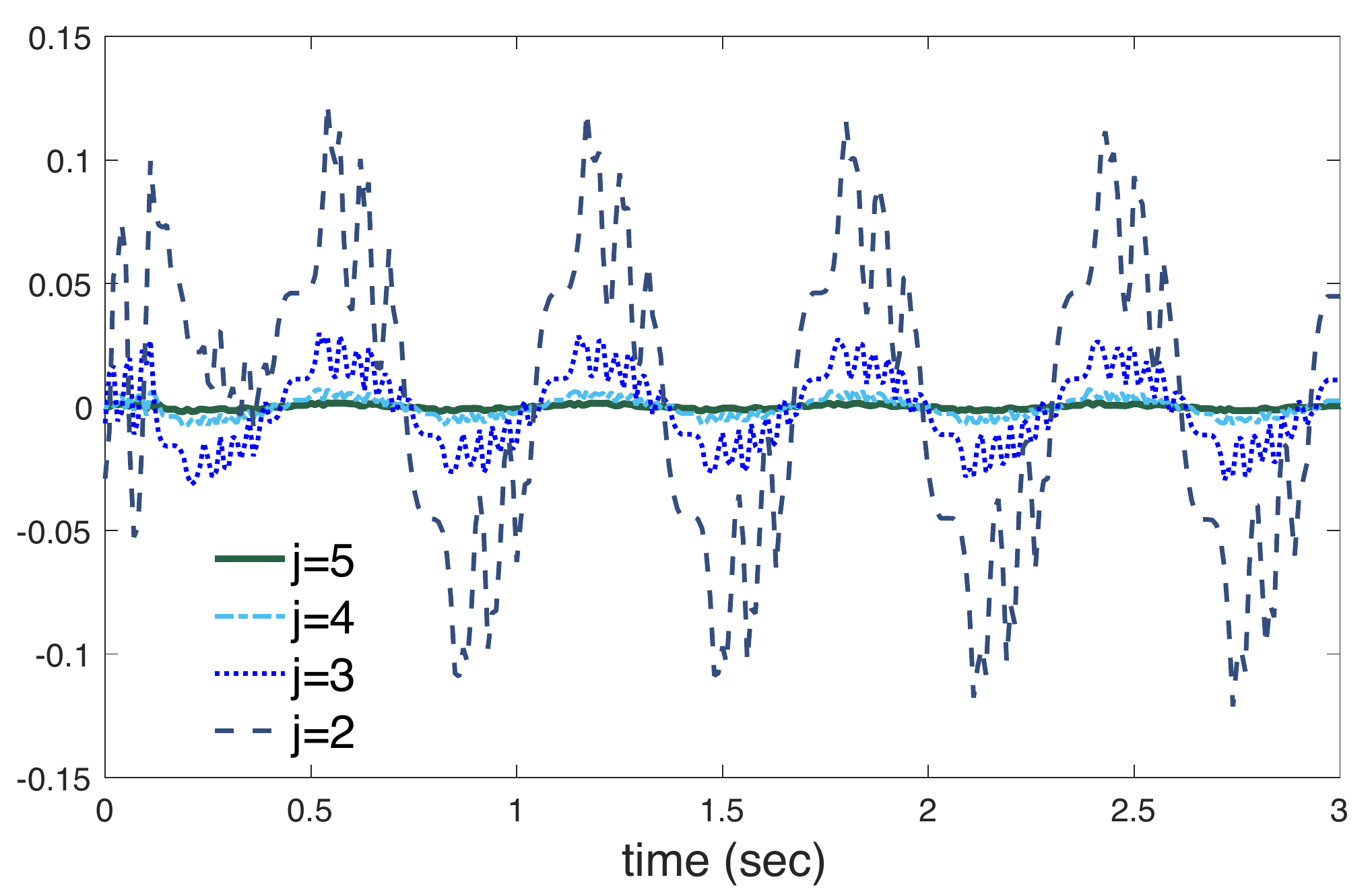}
\end{subfigure}
\caption{Error for different resolution simulations, $J=7$}
\label{fig:ErrorJ7}
\end{figure}
Figure \ref{fig:ErrorJ7} shows the simulation results for $J=7$ and $j=2,3,4,5$. The error term attenuates with increasing j. In order to investigate the rate of attenuation, we evaluate constant $C$ for different levels of refinements. As shown in figure \ref{fig:Cvsj}, $C$ is approximately constant with respect to $j$ which agrees with the result from Equation \ref{eq:numerical}.

\begin{figure}[h!]
\centering
\begin{subfigure}[b]{.04\textwidth}
\includegraphics[width=.65cm]{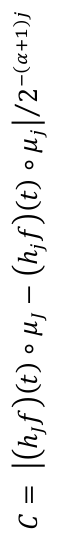}
\end{subfigure}
\begin{subfigure}[b]{.6\textwidth}
\includegraphics[width=9.2cm]{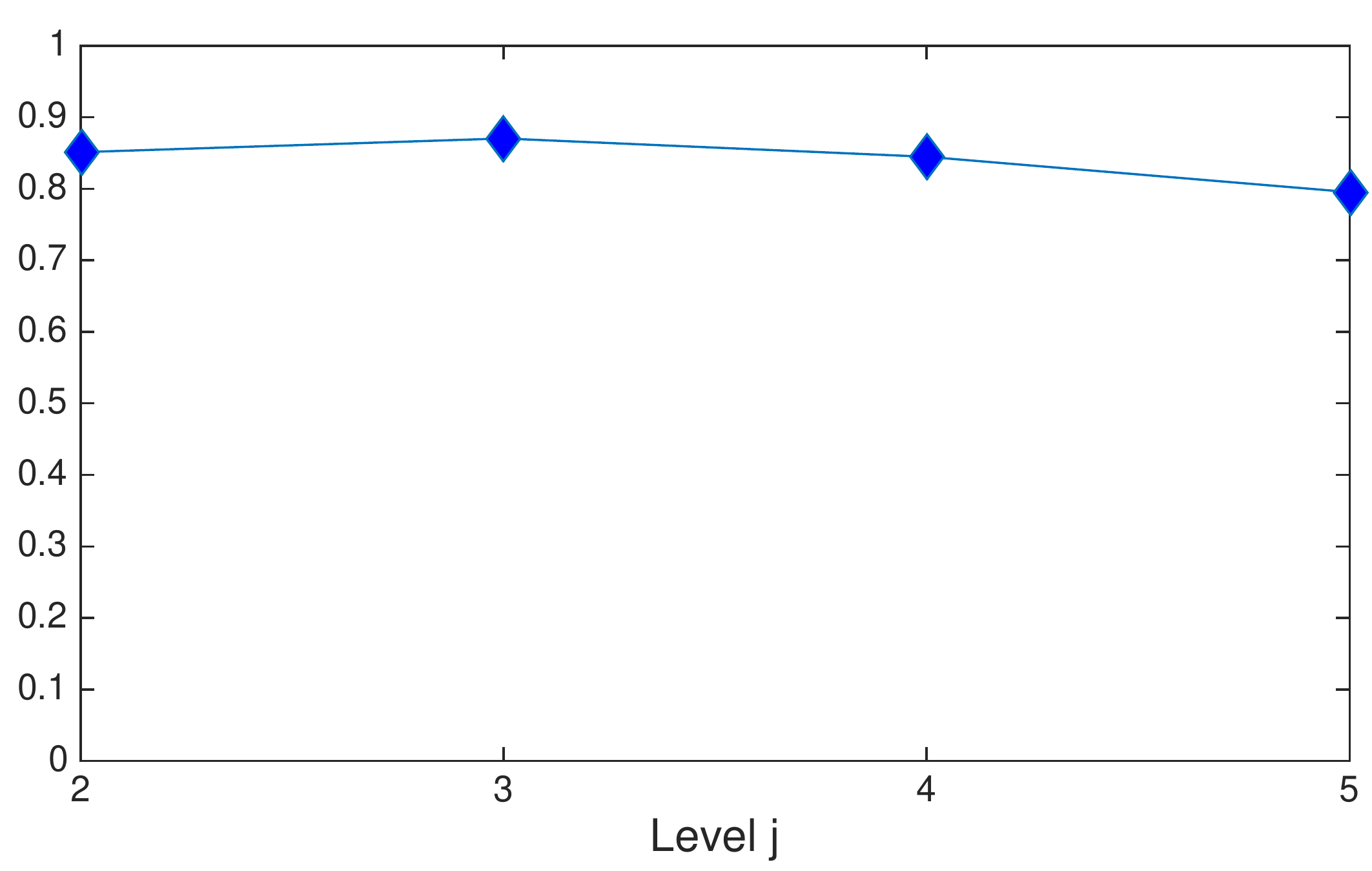}
\end{subfigure}
\caption{$C$ for different level j refinement simulations}
\label{fig:Cvsj}
\end{figure}

\subsection{Online Identification of History Dependent Aerodynamics and Adaptive Control for a Simple Wing Model}
\label{sec:numerics2}
The reformatted governing equations of the system take the form of Equation \ref{eqn:2} where $Q_a(t,\mu)$ is the vector of generalized history dependent aerodynamic loads. The dynamic equation of the system can be written in the form of Equation \ref{eq:1st_order}, where the history dependent term $M^{-1}(q)Q_a(t,\mu)$ is  rewritten in terms of a history dependent operator $(\mathcal{H}X)(t)$ acting on the distributed parameter function $\mu$. The history dependent operator includes a family of fixed history dependent kernels and the distributed parameters $\mu$ act as a weighting vector that determines the contribution of a specific history dependent kernel to the overall history dependent operator. 
\begin{figure}[h!]
\centering
\includegraphics[width=12cm]{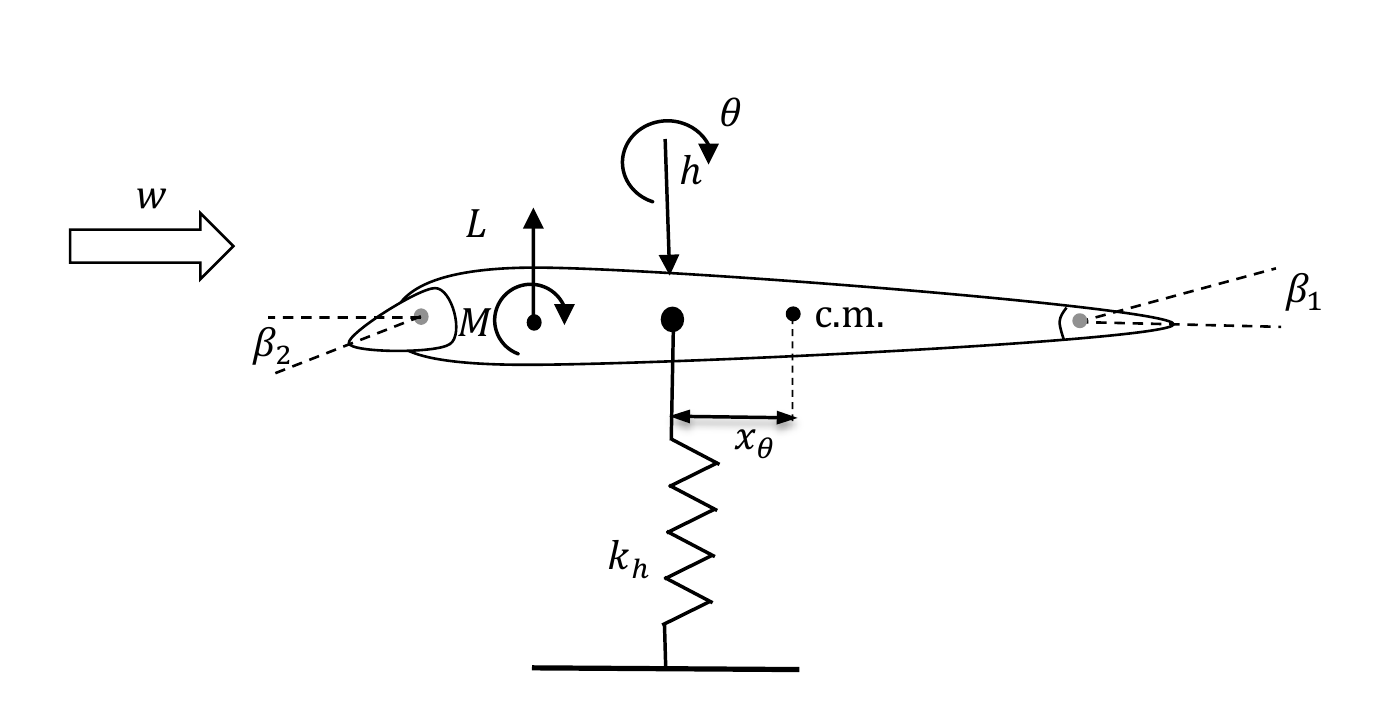}
\caption{Prototypical model for a wing section }
\label{fig:wing_model}
\end{figure}

We perform an offline identification based on a set of experimental data collected from a wind tunnel experiments or CFDsimulations. These define a nominal model for the history dependent aerodynamic loads that appear in the governing equations of the system. We can exploit the model in the numerical simulations to perform an online estimation of the history dependent aerodynamics and adaptive control of a simple wing model. 
The details of offline identification of history dependent aerodynamics follow the steps explained in \cite{Dadashi2016}. 

The model developed in Figure \ref{fig:wing_model}  is chosen to validate our proposed adaptive sliding mode controller where $w$ is the velocity of wind, $k_h$ is spring constant in plunge, $k_\theta$ is a spring constant in pitch, $\theta$ is the pitch angle, $h$ is the plunge displacement, $c_\theta$ and $c_h$ are viscous damping coefficients, $m$ and $I_\theta$ are the mass and moment of inertia and, $x_\theta$ is the non-dimensionalized distance between center of mass and the elastic axis. Finally, $L$ and $M$ are lift and moment generated by the leading and trailing edge flaps. The angles $\beta_1$ and $\beta_2$ define the rotation of the trailing edge and leading edge flaps respectively. The dynamic equations of the wing model is derived in the appendix C as

\begin{equation}
\begin{bmatrix}
m & m x_\theta \\ m  x_\theta  & mx_\theta^2+ I_\theta
\end{bmatrix}
\begin{Bmatrix} \ddot{h} \\ \ddot{\theta}\end{Bmatrix}+ 
\begin{bmatrix} c_h & 0 \\ 0 & c_\theta \end{bmatrix}
\begin{Bmatrix} \dot{h} \\ \dot{\theta}\end{Bmatrix}+
\begin{bmatrix} k_h & 0 \\ 0 & k_\theta \end{bmatrix}
\begin{Bmatrix} h \\ \theta\end{Bmatrix}= \begin{Bmatrix} L \\ 0 \end{Bmatrix} + \begin{Bmatrix} f_1(\beta_1, \beta_2) \\ f_2(\beta_1, \beta_2) \end{Bmatrix}.
\label{eqn:Wing}
\end{equation}
We have assumed the aerodynamic moment $M$ to be zero and the distance  $x_a$  between the aerodynamic center $A$ and hinge point to be negligible to simplify the simulation. The unsteady aerodynamic lift is $ L= Q_a(t,\mu)$ where $Q_a(t,\mu)=(\mathcal{H}X)\circ \mu$ reflects the history dependent nature of aerodynamic loads. We rewrite Equation \ref{eqn:Wing} to achieve the standard form presented in Equation \ref{eqn:2}.  

The adaptive controller presented in Theorem 4 is composed of two parts. The first part compensates for the flutter generated by the history dependent aerodynamic forces through online identification of the aerodynamics. The second part employs an sliding mode controller to compensate for modeling errors.
\begin{figure}[h]
\centering
\begin{subfigure}[b]{.45\textwidth}
\includegraphics[width=7.5cm]{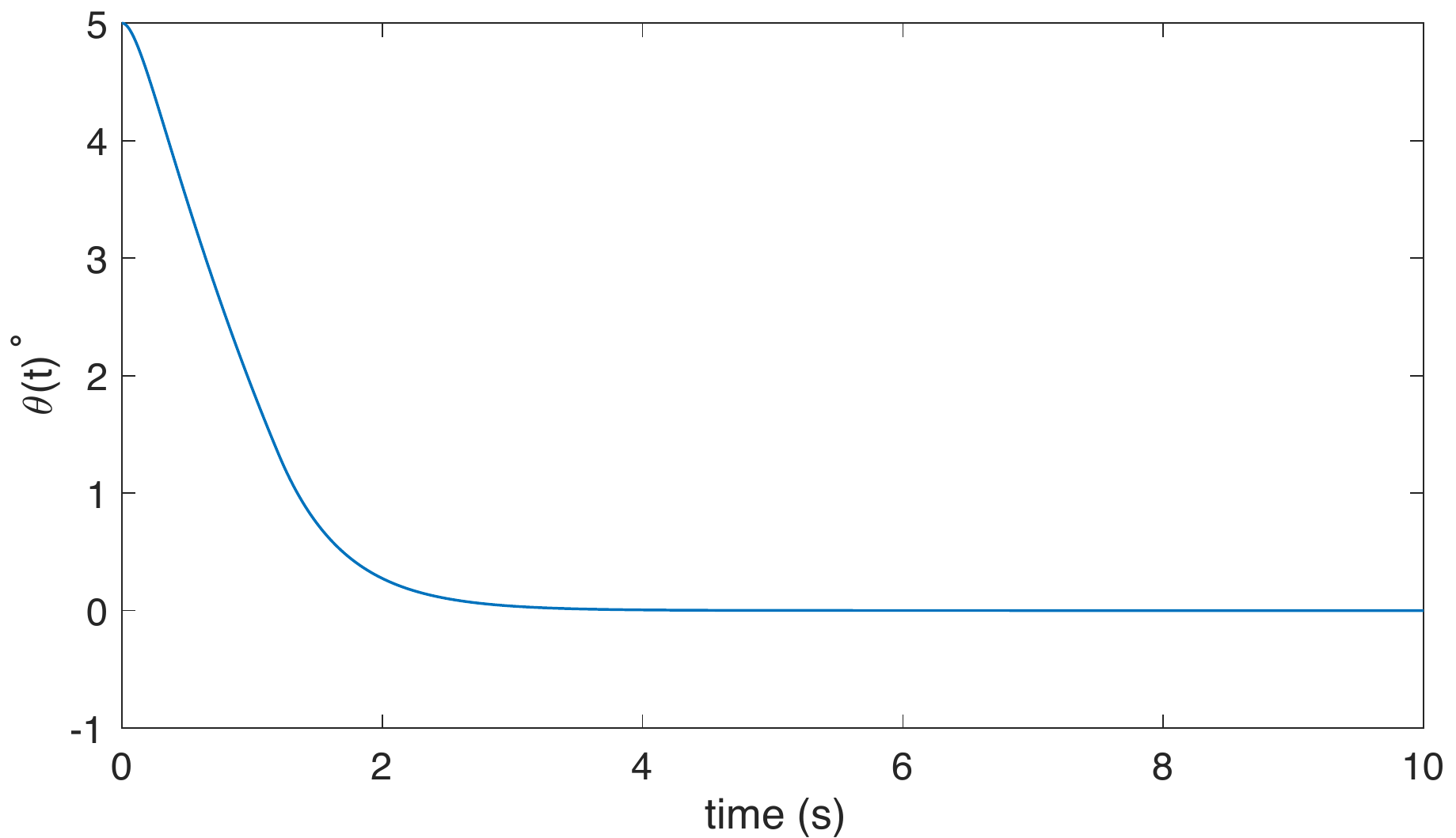}
\subcaption{State trajectory, $\theta(t)$}
\end{subfigure}
\begin{subfigure}[b]{.45\textwidth}
\includegraphics[width=7.7cm]{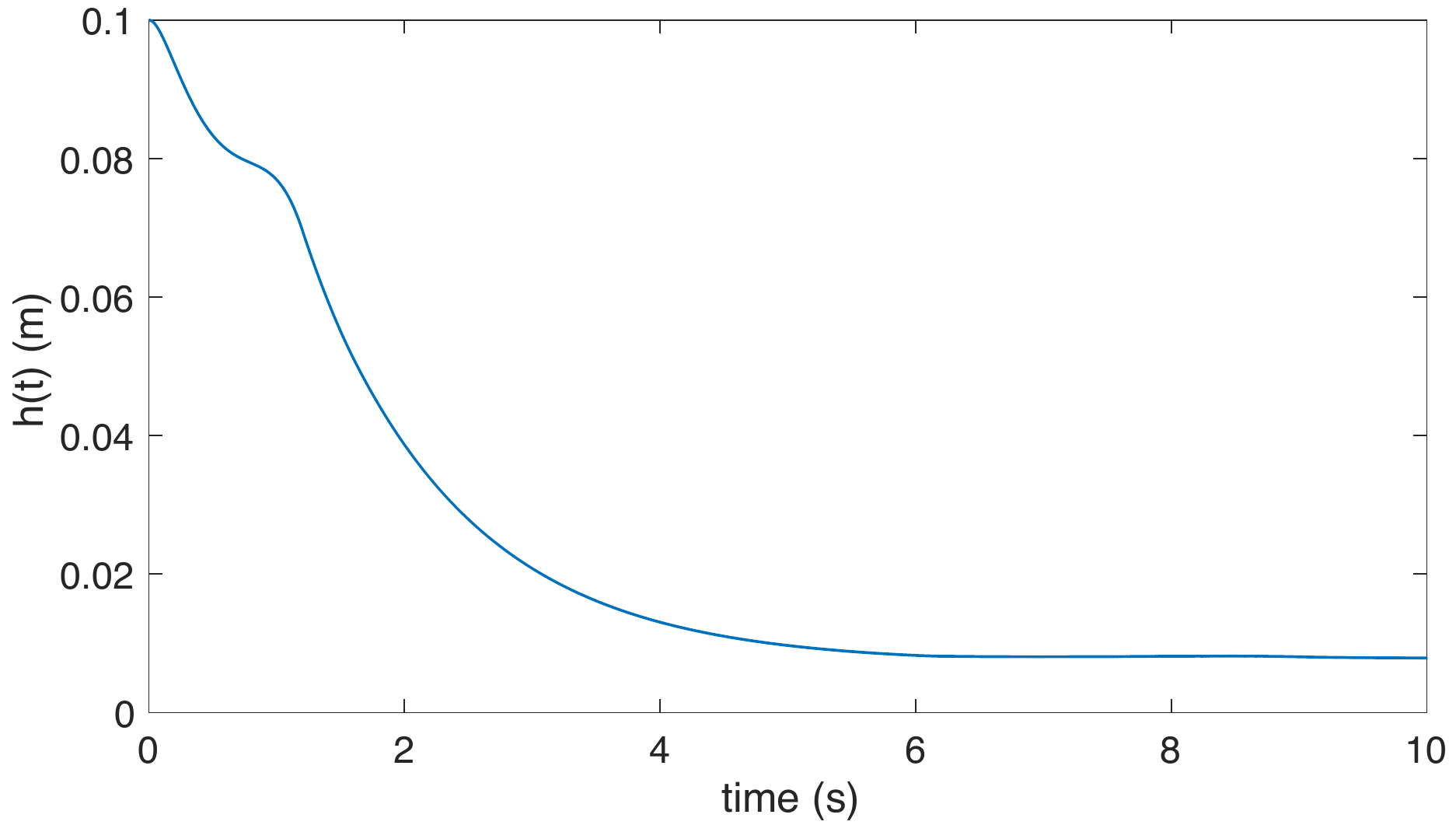}
\subcaption{State trajectory, $h(t)$}
\end{subfigure}
\begin{subfigure}[b]{.45\textwidth}
\includegraphics[width=7.5cm]{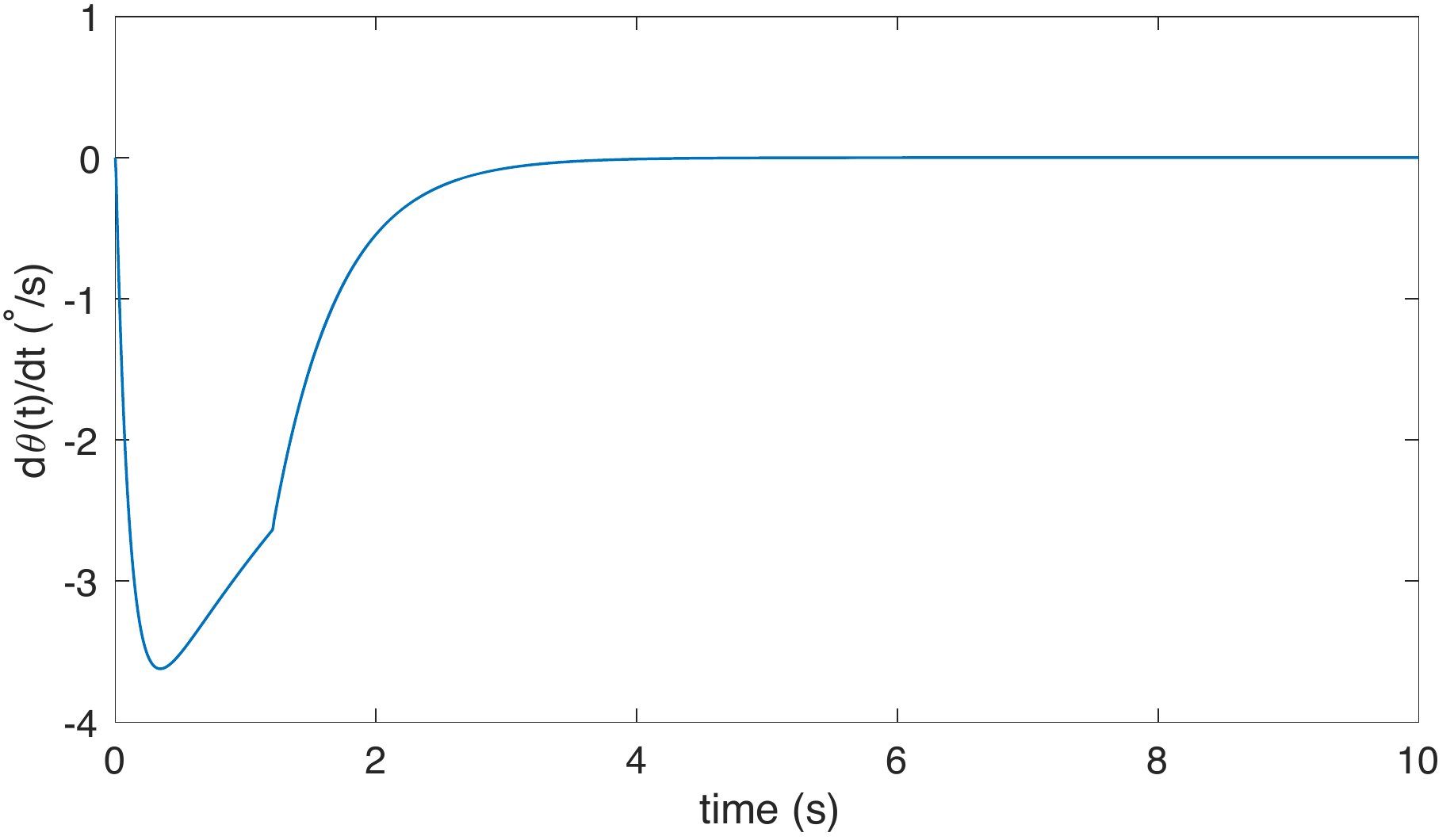}
\subcaption{State trajectory, $\dot{\theta}(t)$}
\end{subfigure}
\begin{subfigure}[b]{.45\textwidth}
\includegraphics[width=7.7cm]{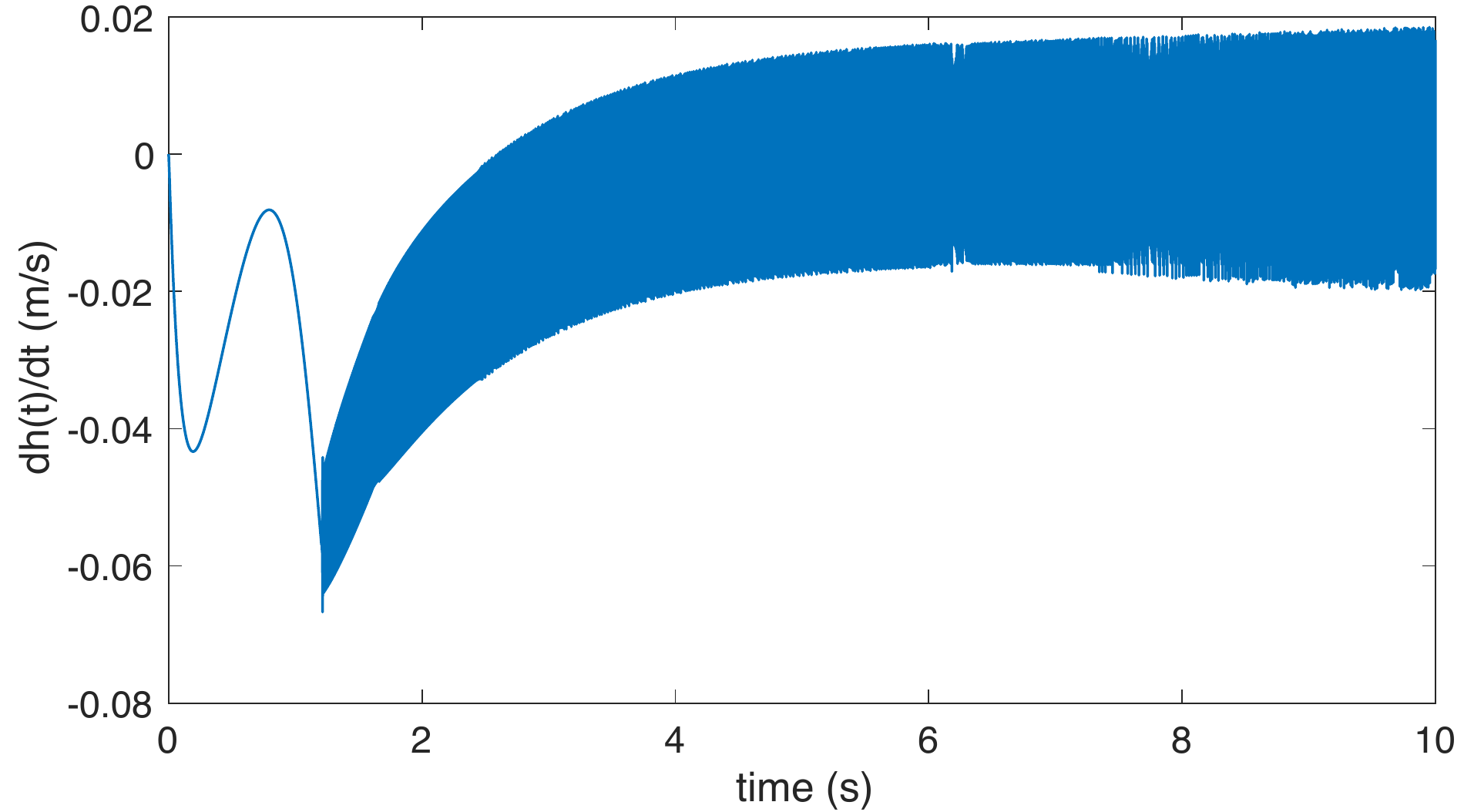}
\subcaption{State trajectory, $\dot{h}(t)$}
\end{subfigure}
\begin{subfigure}[b]{.45\textwidth}
\includegraphics[width=7.5cm]{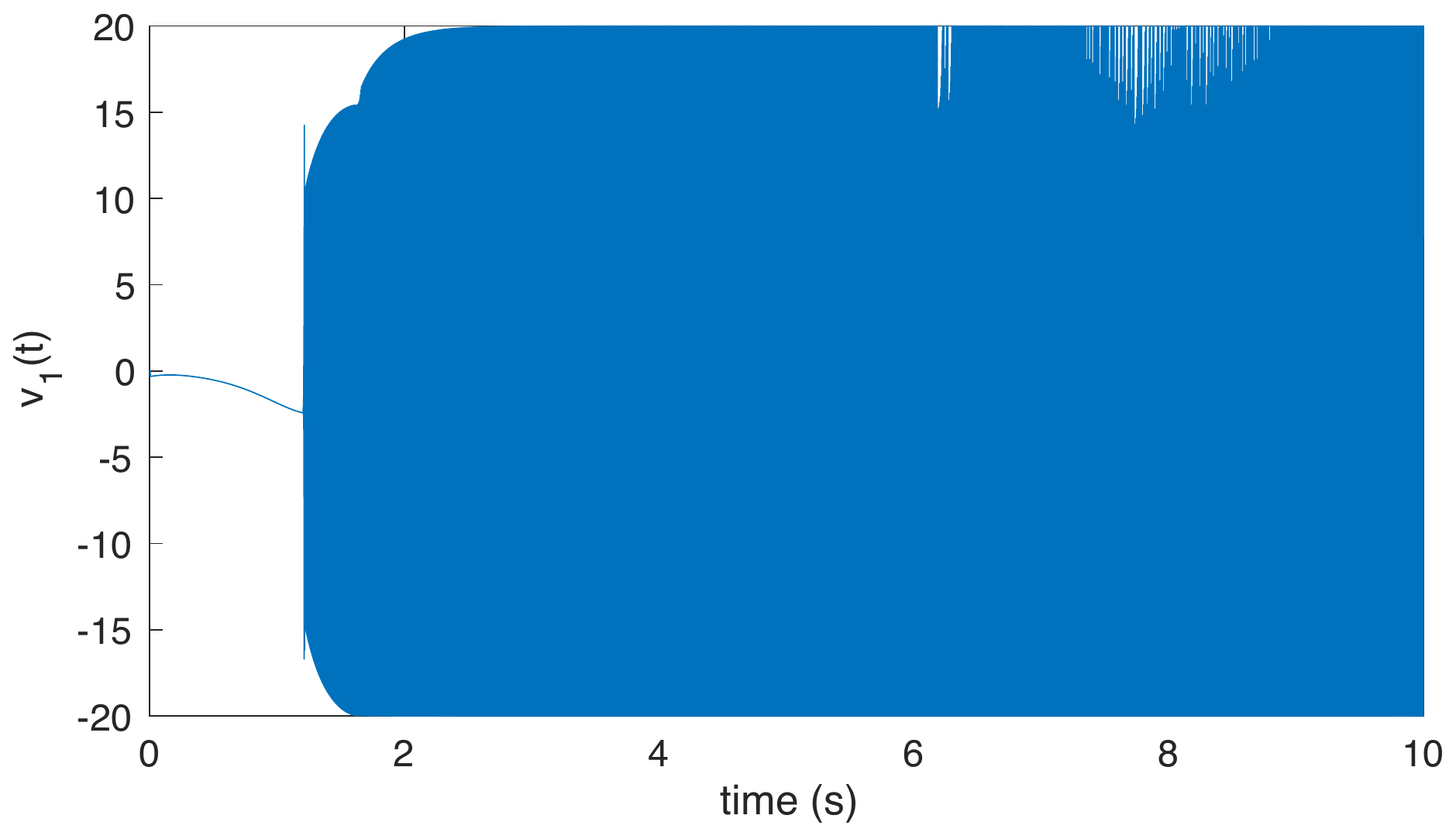}
\subcaption{Sliding mode control input signal, $v_1(t)$}
\end{subfigure}
\begin{subfigure}[b]{.45\textwidth}
\includegraphics[width=7.7cm]{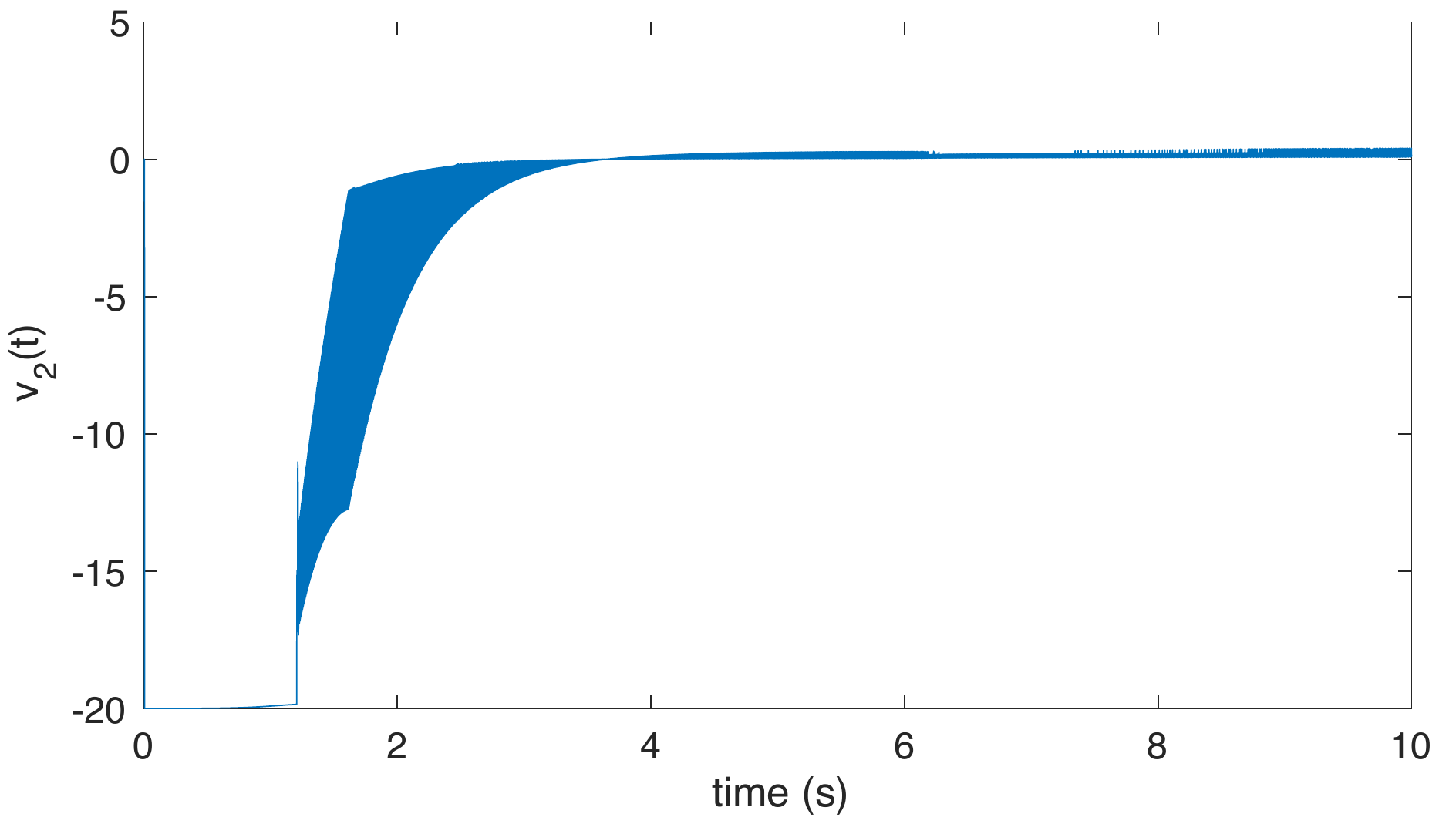}
\subcaption{Sliding mode control input signal, $v_2(t)$}
\end{subfigure}
\caption{Time histories of the states and input signals for $\epsilon=0.01$, $t_h=0.001$(sec) and, $k=20$ }
\label{fig:control_results_chattering}
\end{figure}

It is noteworthy that numerical time integration of the evolution equations must accommodate history dependent terms. Since the dynamics of such systems are given via functional differential equations, the ordinary integration rules are not directly applicable. We exploit the predictor-corrector integration rule that has been introduced first in \cite{tavernini}. We also refer the interested reader to our previous paper \cite{dadashi2016_CDC} for details of such integration rules.

Figure \ref{fig:control_results_chattering} Shows the simulation results for the case where $\epsilon =0.01$ and $t_h=0.001$. The system response eventually enters in a $\epsilon$ neighborhood of the sliding manifold. However, as depicted in the figure a chattering behavior occurs in the control signal and system trajectories. We trace this behavior back to the integration error induced by the size of time step.  When we increase $\epsilon$ or reduce the integration time step, the control signal and system trajectories become smooth. The simulation results for $\epsilon = 0.01$ and $t_h=0.0005 $ are depicted in Figure \ref{fig:control_results_smooth1}.  The system trajectories converge to a neighborhood of zero or the set $\mathcal{M}$ in Equation \ref{eq:M} with time and the control signals are relatively smooth. Also, Figure \ref{fig:control_results_smooth2} shows the case when $\epsilon = 0.1$ and $t_h=0.001 $. The convergence rate of the signals to zero is slower but the results do not show any chattering. Therefore, the proposed smooth sliding mode adaptive controller proves to be effective to identify and compensate for the unknown history dependent aerodynamic forces.  

\begin{figure}[h!]
\centering
\begin{subfigure}[b]{.45\textwidth}
\includegraphics[width=7.5cm]{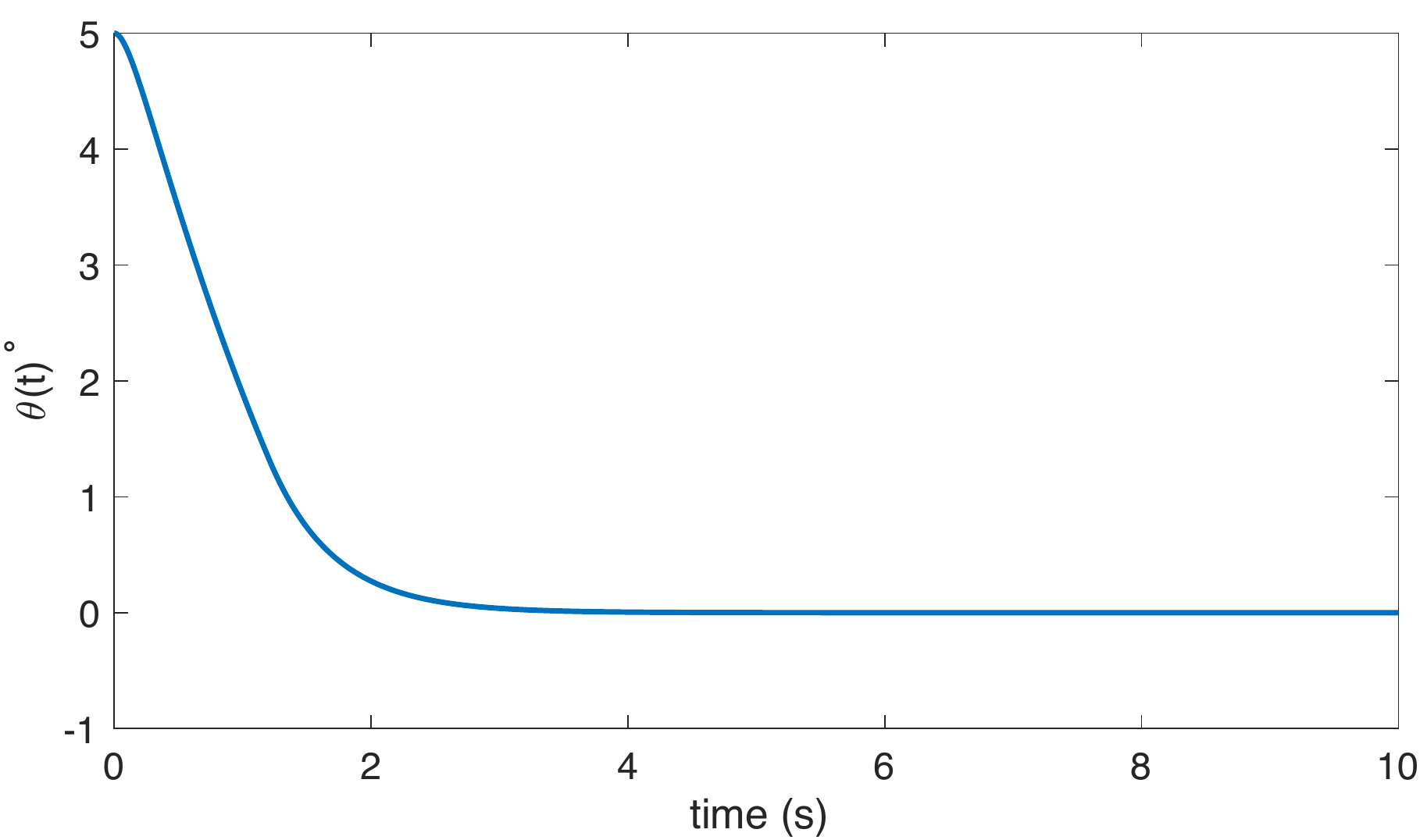}
\subcaption{State trajectory, $\theta(t)$}
\end{subfigure}
\begin{subfigure}[b]{.45\textwidth}
\includegraphics[width=7.7cm]{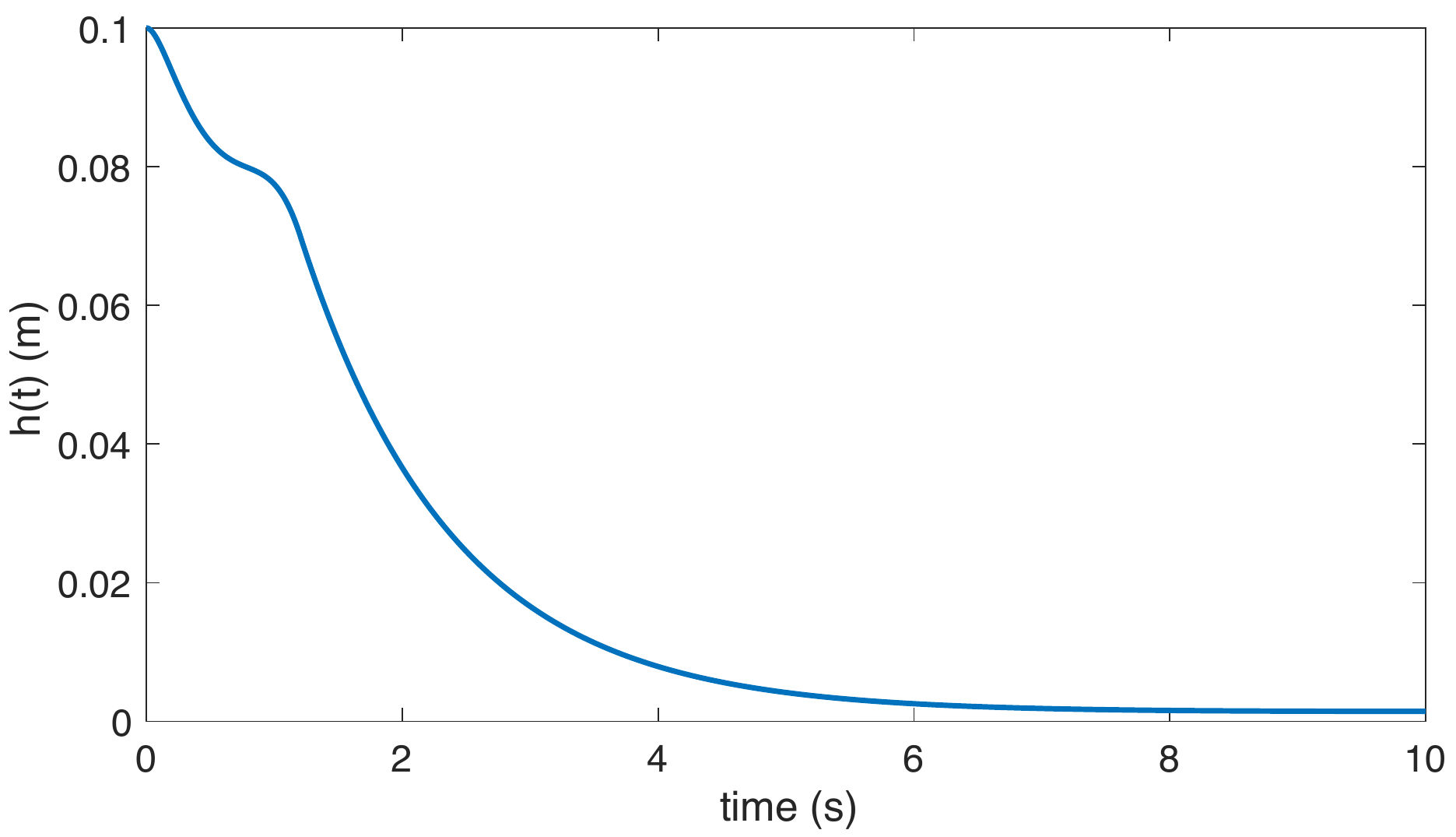}
\subcaption{State trajectory, $h(t)$}
\end{subfigure}
\begin{subfigure}[b]{.45\textwidth}
\includegraphics[width=7.5cm]{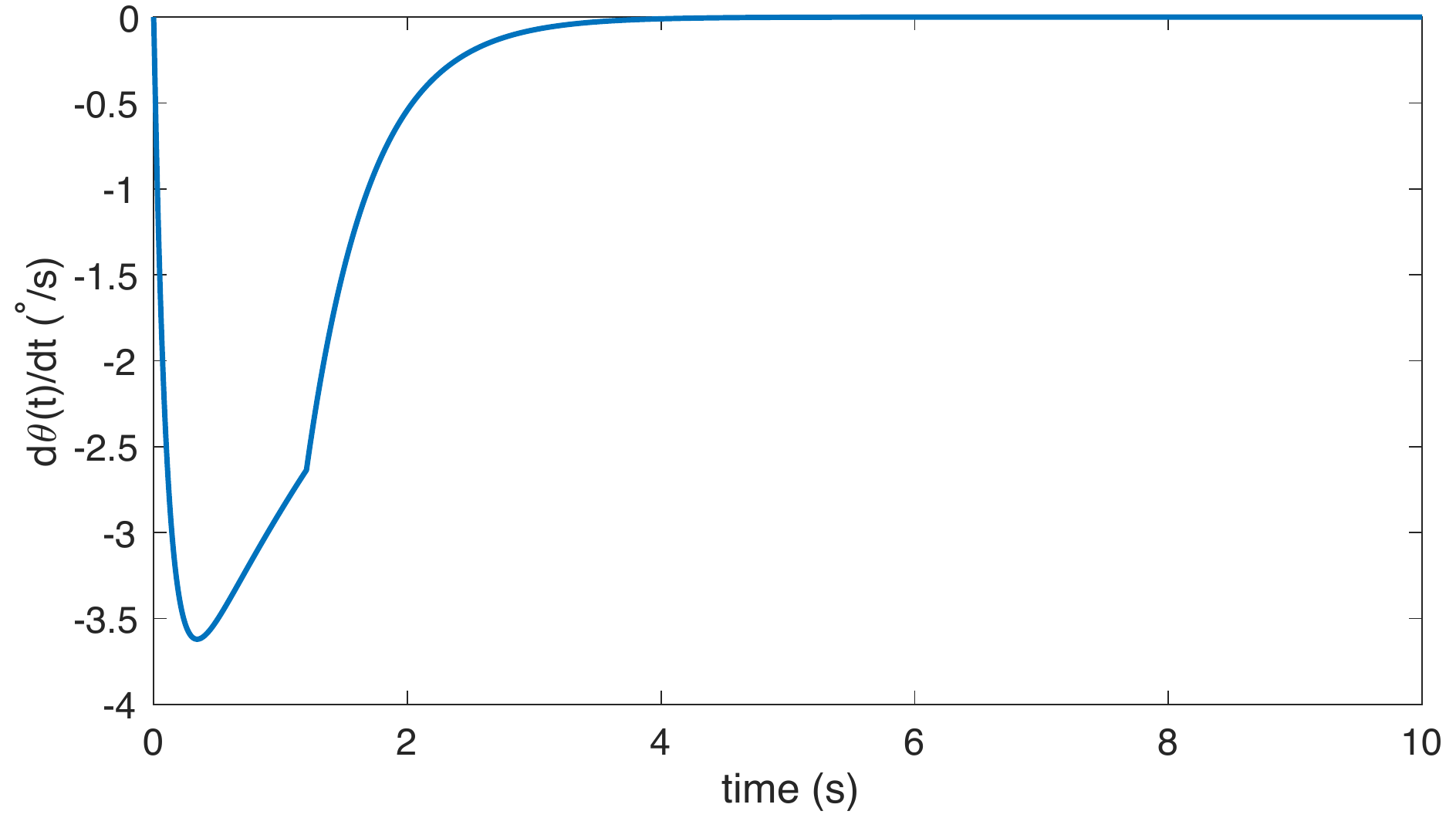}
\subcaption{State trajectory, $\dot{\theta}(t)$}
\end{subfigure}
\begin{subfigure}[b]{.45\textwidth}
\includegraphics[width=7.7cm]{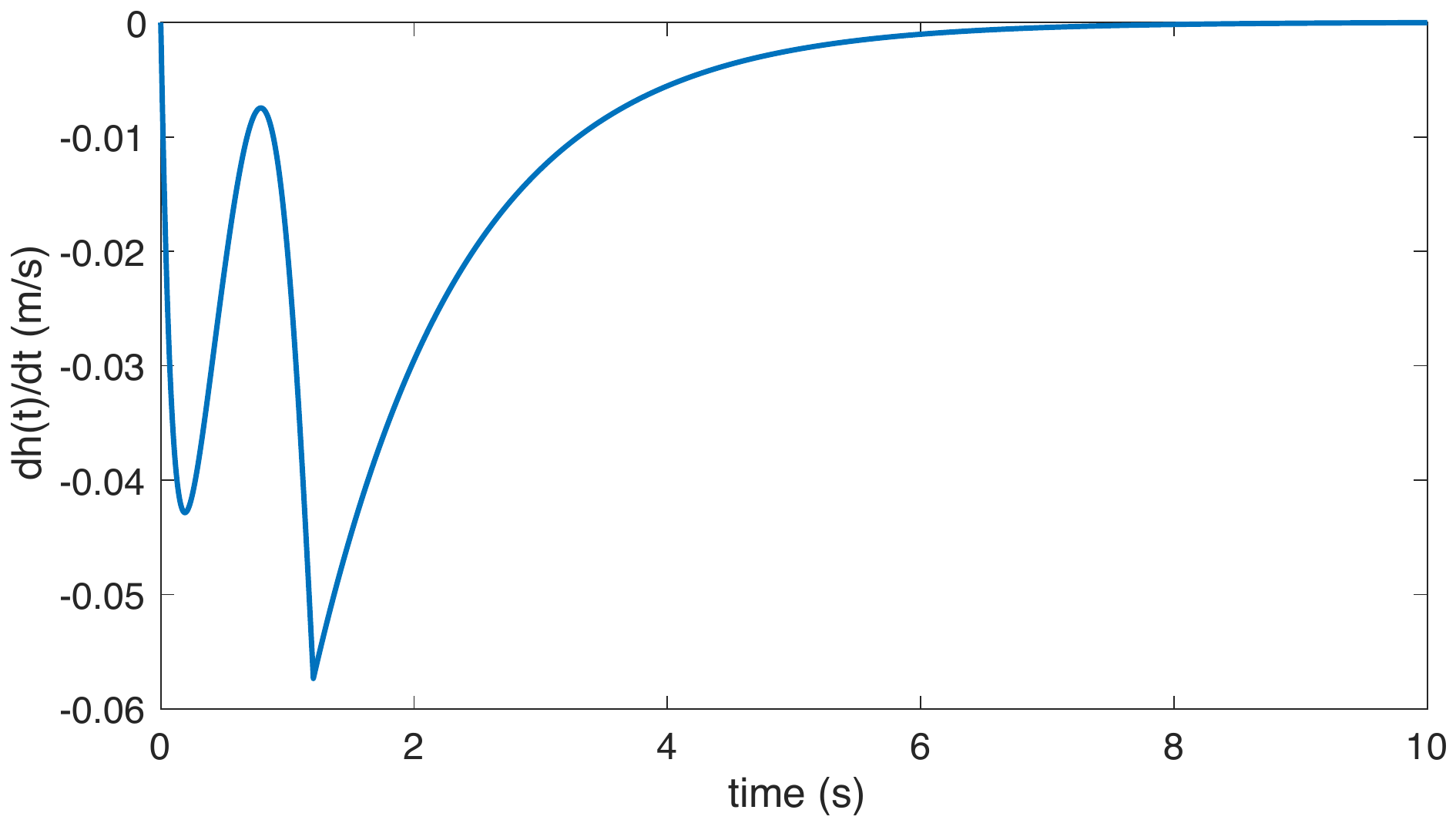}
\subcaption{State trajectory, $\dot{h}(t)$}
\end{subfigure}
\begin{subfigure}[b]{.45\textwidth}
\includegraphics[width=7.5cm]{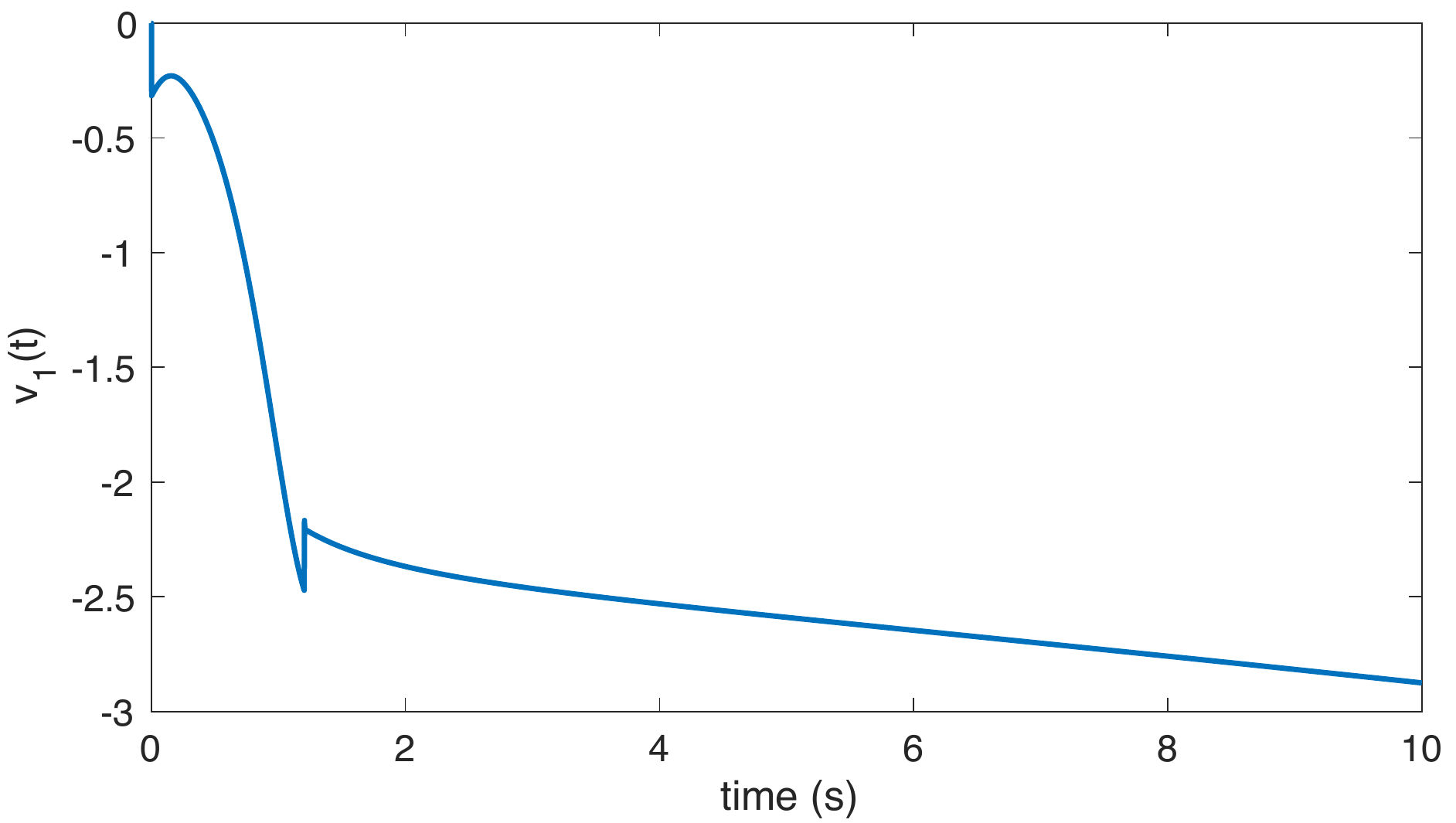}
\subcaption{Sliding mode control input signal, $v_1(t)$}
\end{subfigure}
\begin{subfigure}[b]{.45\textwidth}
\includegraphics[width=7.7cm]{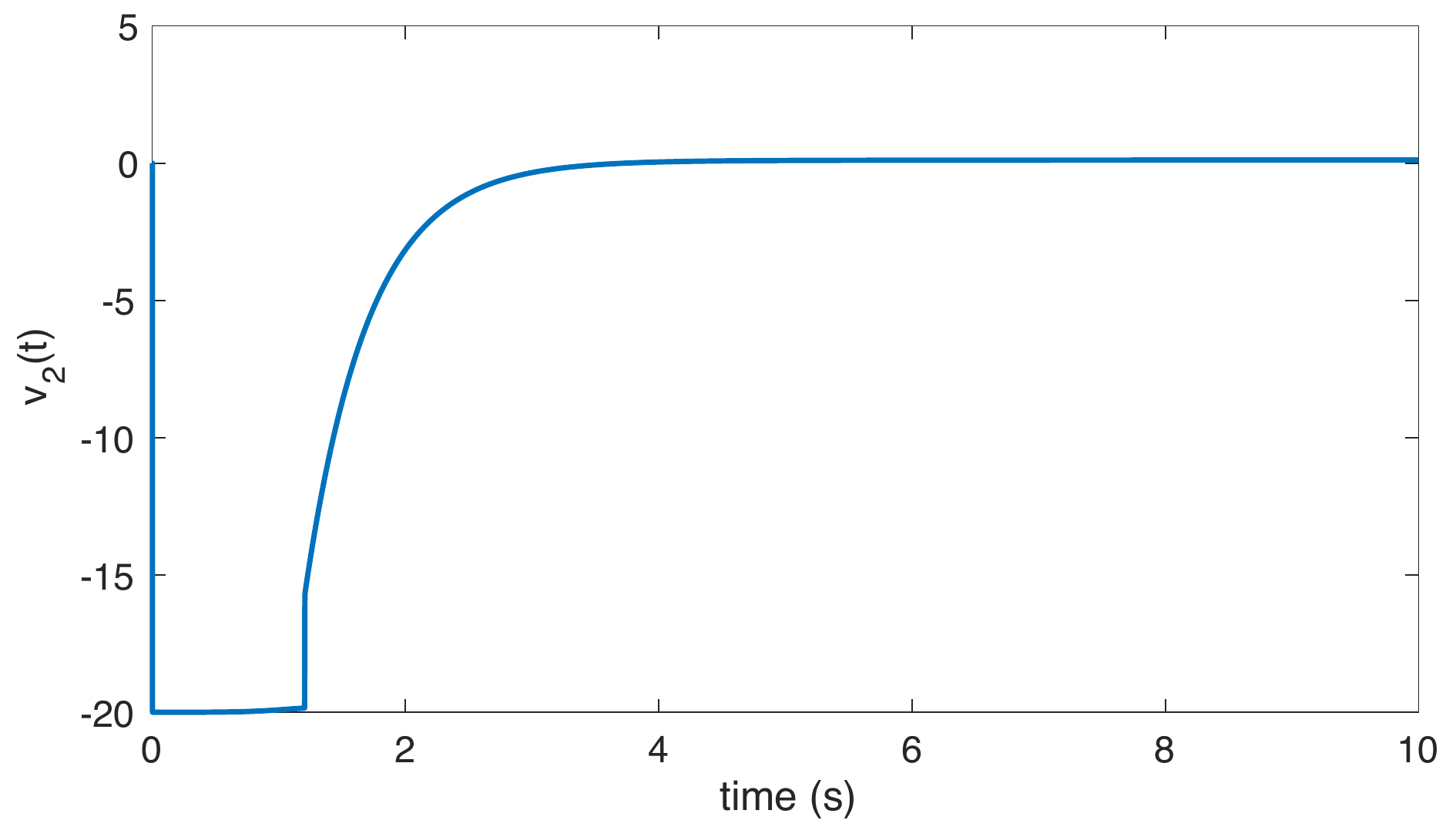}
\subcaption{Sliding mode control input signal, $v_2(t)$}
\end{subfigure}
\caption{Time histories of the states and input signals for $\epsilon=0.01$, $t_h=0.0005$ (sec) and, $k=20$ }
\label{fig:control_results_smooth1}
\end{figure}

\begin{figure}[h!]
\centering
\begin{subfigure}[b]{.45\textwidth}
\includegraphics[width=7.5cm]{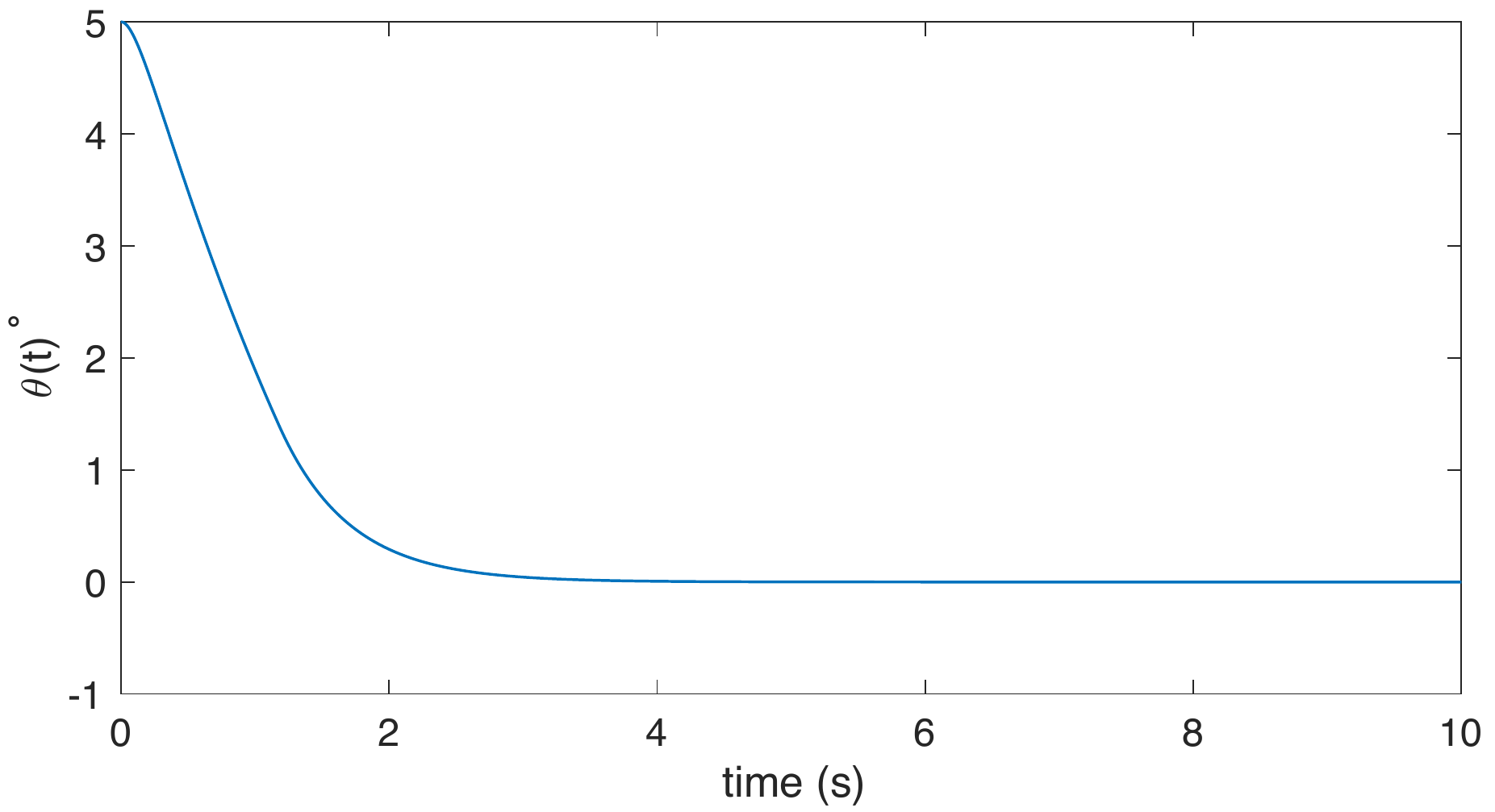}
\subcaption{State trajectory, $\theta(t)$}
\end{subfigure}
\begin{subfigure}[b]{.45\textwidth}
\includegraphics[width=7.7cm]{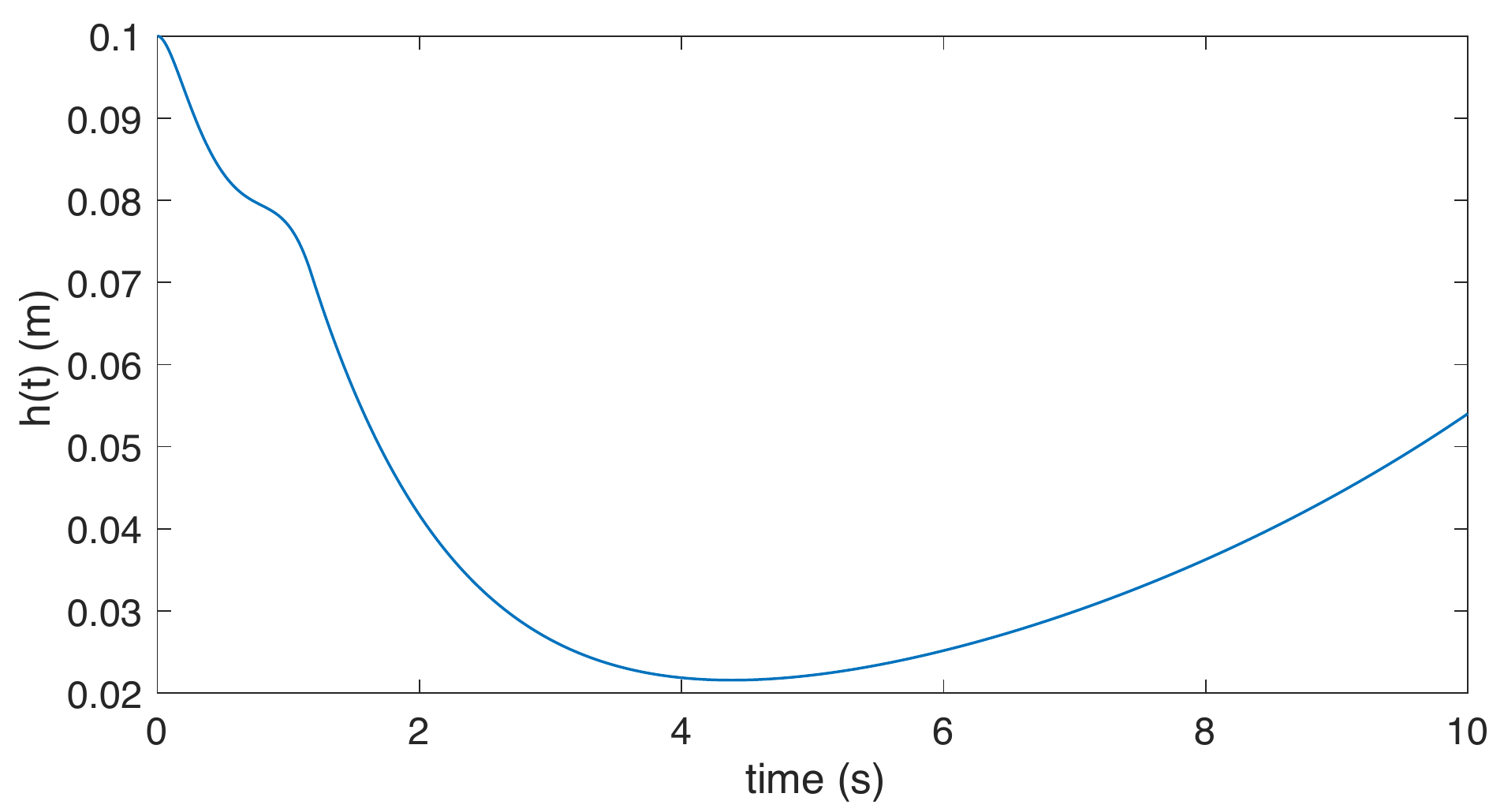}
\subcaption{State trajectory, $h(t)$}
\end{subfigure}
\begin{subfigure}[b]{.45\textwidth}
\includegraphics[width=7.5cm]{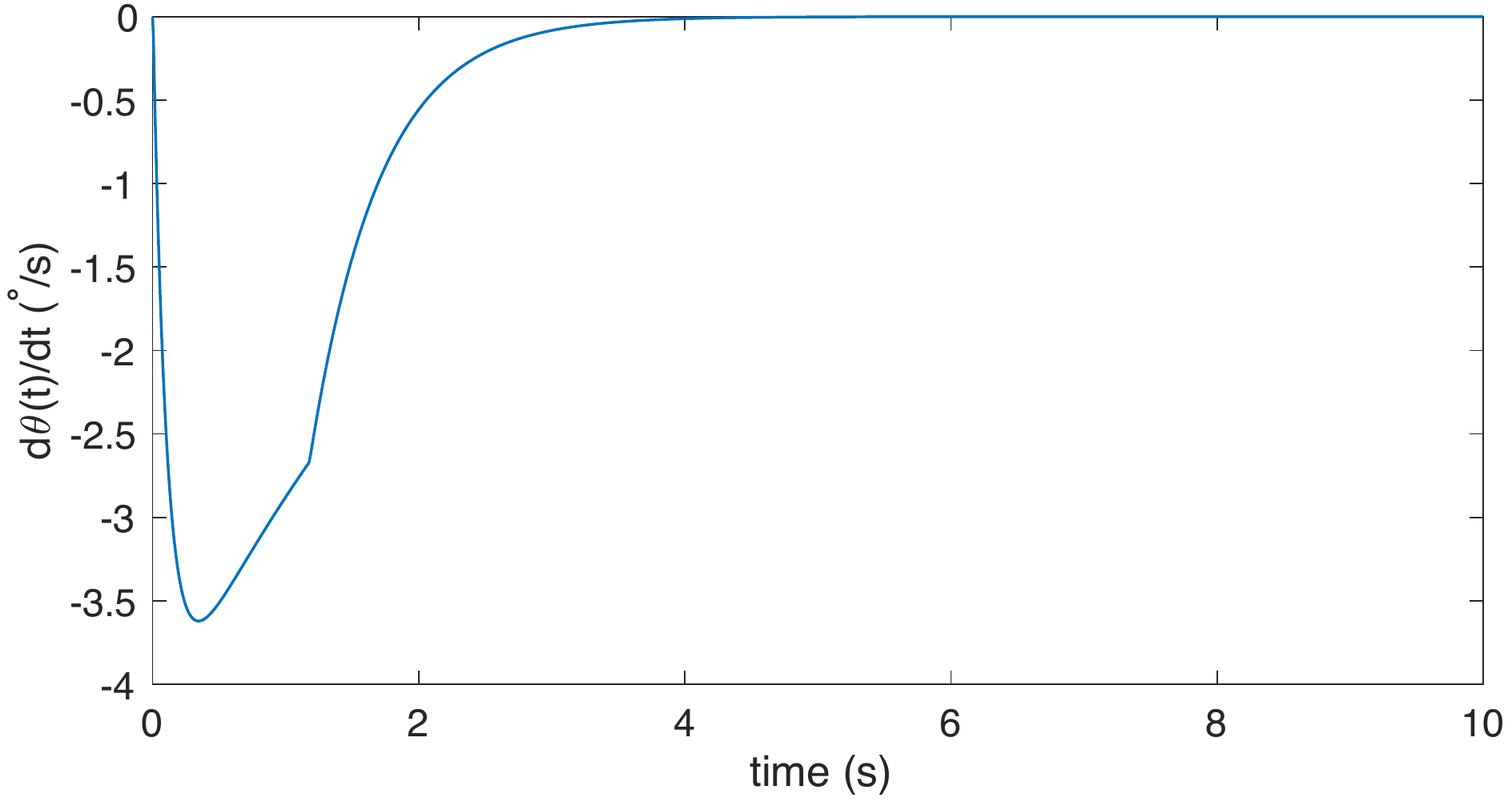}
\subcaption{State trajectory, $\dot{\theta}(t)$}
\end{subfigure}
\begin{subfigure}[b]{.45\textwidth}
\includegraphics[width=7.7cm]{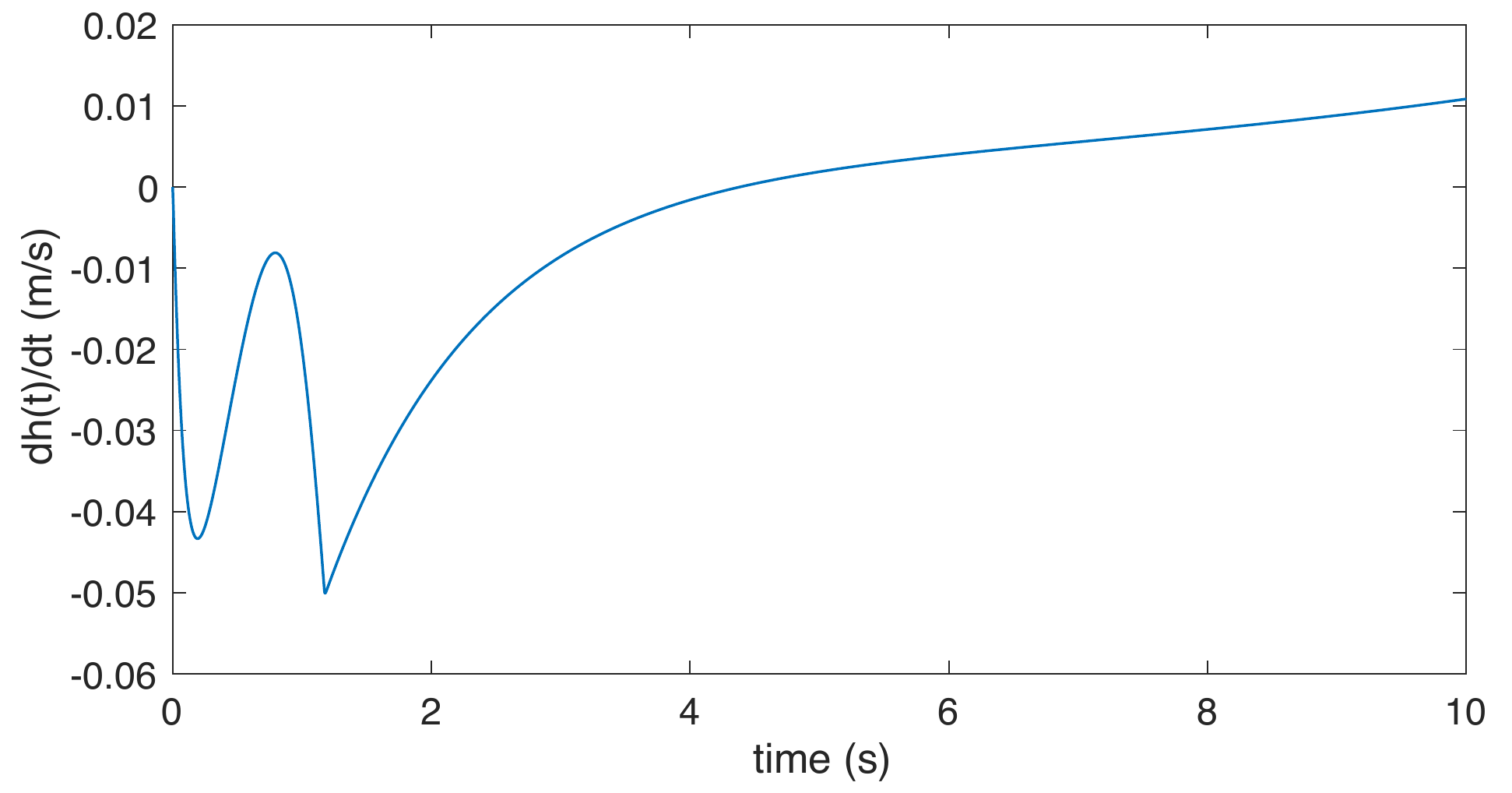}
\subcaption{State trajectory, $\dot{h}(t)$}
\end{subfigure}
\begin{subfigure}[b]{.45\textwidth}
\includegraphics[width=7.5cm]{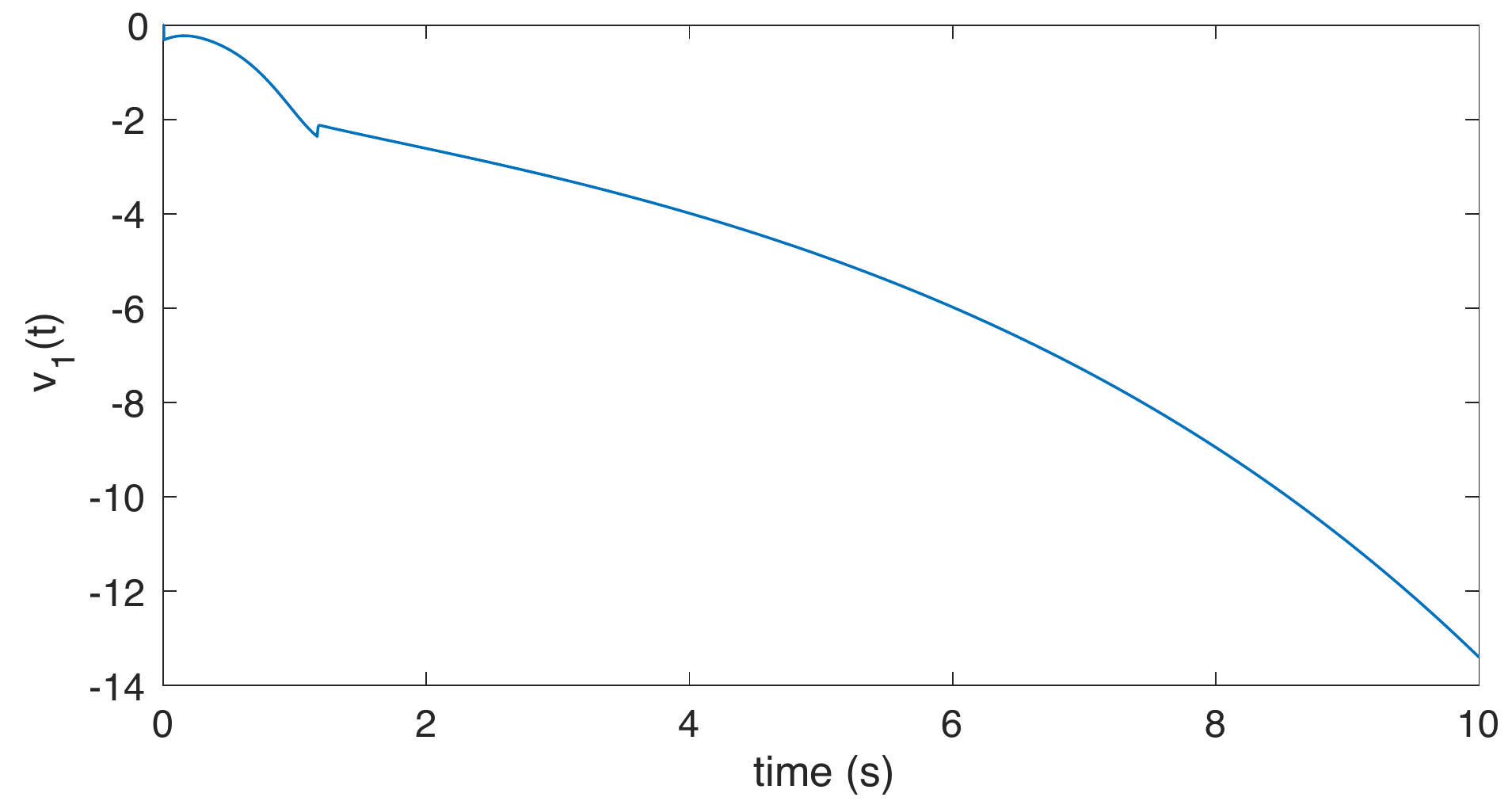}
\subcaption{Sliding mode control input signal, $v_1(t)$}
\end{subfigure}
\begin{subfigure}[b]{.45\textwidth}
\includegraphics[width=7.7cm]{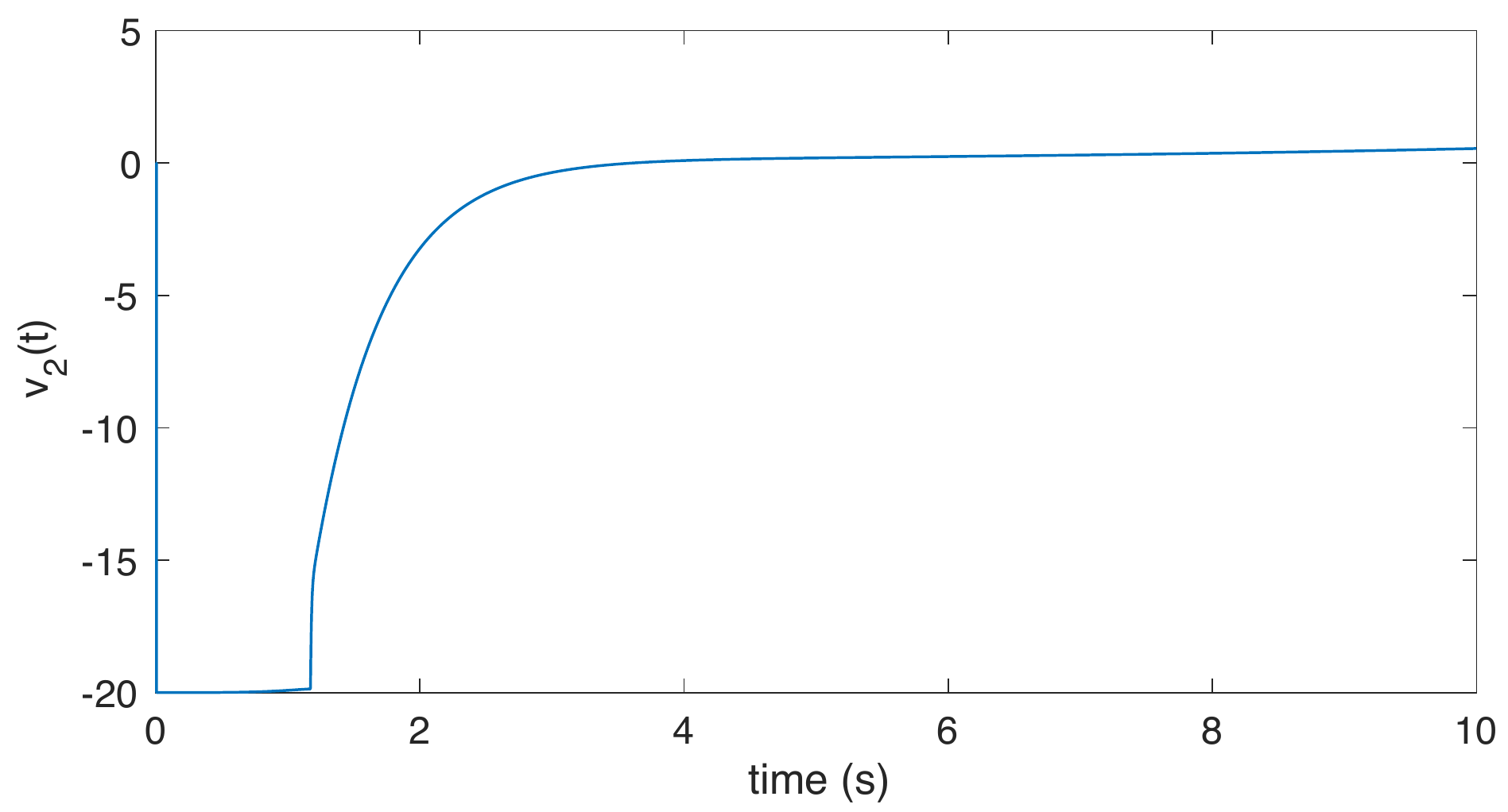}
\subcaption{Sliding mode control input signal, $v_2(t)$}
\end{subfigure}
\caption{Time histories of the states and input signals for $\epsilon=0.1$, $t_h=0.001$ (sec) and, $k=20$ }
\label{fig:control_results_smooth2}
\end{figure}

\section{Results and Conclusion}
\nocite{*}
In this paper, we have derived an explicit bound for the error of approximation for certain history dependent operators that are used in construction of robotic FDE's in ~\cite{Dadashi2016} and this paper. The numerical simulations presented validate our results. We establish uniform upper bounds on their accuracy of the approximations.  The uniform $\mathcal{O}(2^{-(\alpha+1)j})$ rates of approximation for grid resolution $j$ depend on the Holder coefficient $\alpha$ that describes the smoothness of the ridge functions that define the history dependent kernels.
In Section \ref{sec:exist} we prove the existence and uniqueness of a local solution for the special case of functional differential equations with history dependent terms shown in Equation \ref{eq:first_order}. Since the functional differential equation of interest evolves in an infinite dimensional space, we construct finite dimensional approximations with grid resolution $j$. We further show that the solution of the finite dimensional distributed parameter system converges to the solution of the infinite dimensional FDE as the resolution is refined. Finally, we propose an adaptive control strategy to identify and compensate the unknown history dependent dynamics.

\section*{Appendix A: Wavelets and Approximation Spaces over the Triangular Domain}
\label{App:B}
We define the multiscaling functions 
$$\phi_{j,k}(x) =1_{\Delta_{i_1,i_2,\hdots,i_j}}(x)/\sqrt[]{m(\Delta_{i_1,i_2,\hdots,i_j})}
$$
in which 
\begin{align*}
1_{\Delta_{s}}(x)=\left \{
\begin{array}{lll}
1 & x \in \Delta_s\\
0 & \text{otherwise}
\end{array}
\right.
\end{align*}
and $m(\Delta_{i_1,i_2,\hdots,i_j})$ is the area of a triangle in the level $j$ refinement.
We have defined
$
(hf)(t)\circ \mu= \iint_\Delta \kappa(s,t,f)\mu(s)ds.
$
The approximation  $(h_j f)(t)\circ \mu$ of this operator is given by
\begin{align*}
(h_j f)(t)\circ \mu= \iint_\Delta \sum_{l\in \Gamma_j} 1_{\Delta_{j,l}}(s) \kappa(\xi_{j,l},t,f) \mu(s)ds,
\end{align*}
where $\xi_{j,l}$ is the quadrature point of number $l$ triangle of grid level $j$.  We approximate $\mu(s) \approx \sum_{m\in \Gamma_j} \mu_{j,m}\phi_{j,m}(s)$. 
Therefore,
\begin{align*}
(& h_j f)(t)\circ \mu_j\\ & =\iint_S \left( \sum_{l\in \Gamma_j} 1_{\Delta_{j,l}}(s) \kappa(\xi_{j,l},t,f) \sum_{m\in \Gamma_j} \mu_{j,m}\phi_{j,m}(s)\right) ds\\
& =\sum_{l\in \Gamma_j} \sum_{m\in \Gamma_j}\kappa(\xi_{j,l},t,f) \left(\iint_S 1_{\Delta_{j,l}}(s) \phi_{j,m}(s) ds\right) \mu_{j,m}\\
& =\sum_{l\in \Gamma_j} \kappa(\xi_{j,l},t,f) \sqrt{m(\Delta_{j,l})}\mu_{j,l}.
\end{align*}
\vspace{-1mm}
For an orthonormal  basis $\left \{ \phi_k \right \}_{k=1}^\infty$ of the separable Hilbert space $P$, we define the finite dimensional  spaces for constructing approximations as $P_n:=\text{span}\left  \{  \phi_k\right \}_{k=1}^n$. The approximation error $E_n$  of $P_n$ is given by 
$$
E_n(f):= \inf_{g\in P_n} \| f-g\|_P.
$$
The approximation space $\mathcal{A}^\alpha_2$ of order $\alpha$ is defined as the collection of functions in $P$ such that
$$
\mathcal{A}^\alpha_2 := \biggl \{
f\in P \biggl |    |f|_{\mathcal{A}^\alpha_2} 
:= \left \{ \sum_{n=1}^\infty (n^\alpha E_n(f))^2 \frac{1}{n} \right \}^{1/2} < \infty
\biggr\}.
$$
For our purposes, the approximation spaces are easy to characterize:  they consist of all functions $f \in P$ whose generalized Fourier coefficients decay sufficiently fast.  That is, $f\in \mathcal{A}^\alpha_2$ if and only if 
$$
\sum_{k=1}^\infty k^{2\alpha}|(f,\phi_k)|^2 \leq C
$$
for some constant $C$.

\section*{Appendix B: The Projection Operator $\Phi_{J\rightarrow j}$}
\label{App:B}
The orthogonal projection operator $\Phi_{J\rightarrow j}: V_J\rightarrow V_j$ maps a distributed parameter  $\mu_J$ to $\mu_j$ i.e. $\Phi_{J\rightarrow j}:\mu_J \mapsto \mu_j$.
\begin{figure}[h!]
\centering
\includegraphics[width=0.2\textwidth]{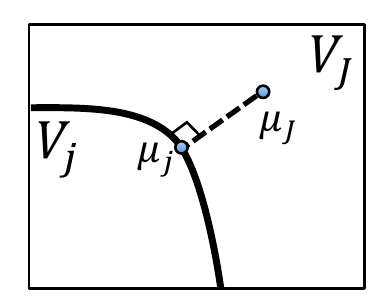}
\caption{Projection Operator $\Phi_{J\rightarrow j}:V_J\rightarrow V_j$}
\label{fig:Proj}
\end{figure}
By exploiting the orthogonality property of the operator we have
\begin{align*}
\iint_\Delta \left( \sum_{m\in\Gamma_j} \mu_{j,m}\phi_{j,m}(s)-\sum_{l\in\Gamma_J} \mu_{J,l}\phi_{J,l}(s)\right) \phi_{j,n}(s) ds=0.
\end{align*}
Therefore, we can write
\begin{align*}
\sum_{m\in\Gamma_j}& \left( \iint_\Delta \phi_{j,m}(s)\phi_{j,n}(s) ds \right)\mu_{j,m}  =\sum_{l\in\Gamma_J} \left( \iint_\Delta \phi_{J,l}(s)\phi_{j,n}(s) ds \right) \mu_{J,l}.
\end{align*}
Since orthogonality implies 
$\iint_\Delta \phi_{j,m}(s)\phi_{j,n}(s) ds=\delta_{m,n},$
we conclude that
 \begin{align*}
\mu_{j,n}=\sum_{l\in\Gamma_j} \left(\iint_\Delta \phi_{j,n}(s)\phi_{J,l}(s) ds\right)\mu_{J,l}.
\end{align*}
From Theorem 1 we have 
\begin{align*}
|(h_{j}f)(t)\circ \Pi_{j}\mu - (hf)(t)\circ \mu | \leq \tilde{C}2^{-\alpha j}, \end{align*}
with
\begin{align*}
(hf)(t)\circ \mu = \iint_\Delta k(s,t,f)\mu(s)\mathrm{d}s,\\
(h_{j}f)(t)\circ\mu = \iint_\Delta \sum 1_{\Delta_j,l}(s)k(\zeta_{j,l},t,f)\mu(s)\mathrm{d}s,
\end{align*}
where $\mu \in P = L^{2}(\Delta) $  and we approximate $ \mu (s) \approx \sum_{l\in\Gamma_J}\mu_{J,l}\phi_{J,l}(s) \in V_J$. 
To  implement this for the finest grid $J$, we compute
\begin{align*}
(h_J f)(t)\circ \mu_J &= (h_j f)(t)\circ \Pi_J \mu_J,\\
&=\iint\left(\sum 1_{\Delta_J,l}(s)k(\zeta_{J,l},t,f)\sum_{m\in\Gamma_J}\mu_{J,m}\phi_{J,m}(s)\right)\mathrm{d}s,\\
&=\sum_{l\in\Gamma_J}\sum_{m\in\Gamma_J}k(\zeta_{J,l},t,f)\left(\iint 1_{\Delta_J,l}(s)\phi_{J,m}(s)\mathrm{d}s\right)\mu_{J,m},\\
&=\sum_{l\in\Gamma_J}\frac{k(\zeta_{J,l},t,f)\mu_{J,l}}{\left(\sqrt[]{m(\Delta_{J,l}})\right)},
\end{align*}
when  $\sqrt[]{m(\Delta_{J,l})}$ is the area of the corresponding triangle $\Delta_{J,l}$ in the grid having resolution level $J$.

\section*{Appendic C: Gronwall's Inequality}
\label{App:C}
We employ the integral form of Gronwall's Inequality to obtain our final convergence result.  Many forms of Gronwall's Inequality exist, and we will use a particularly simple version.  See Section 3.3.4 in \cite{is2012}.  If the piecewise continuous function $f$ satisfies the inequality
$$
f(t) \leq \alpha(t) + \int_0^t \beta(s) f(s) ds
$$
with some piecewise continuous functions $\alpha,\beta$ where $\alpha$ is nondecreasing, then
$$
f(t) \leq \alpha(t) e^{\int_0^t\beta(s)ds}.
$$

\section*{Appendix D: Modeling of a Prototypical Wing Section}
\label{App:D}
Figure \ref{fig:wing_model} shows a simplified model of the wing. In the figure we denote the center of mass by $c.m.$, $A$ is the aerodynamic center, and $O$ is the elastic axis of the wing. The constants $K_h$ and $K_\theta$ are the linear and torsional stiffness, and $h$ is the distance from origin to point $O$ in the fixed reference frame. We denote by $x_\theta$ the distance between point $O$ and center of mass, whereas $x_a$ is the distance between $O$ and $A$. Point $O$ is the origin for the body fixed reference frame. 


We employ The Euler-Lagrange technique to derive the equation of motion for the depicted wing model. The function  $L(\theta,\dot{\theta})$ is the history dependent lift force acting at the aerodynamic center, and $M(\theta,\dot{\theta})$ is the history dependent aerodynamic moment about point $A$. The variables $L_{\beta_1}$ and  $L_{\beta_2}$ are the actuating forces acting at point $D$, and $\beta_1$, $\beta_2$ are the angles between the midchord of the  wing and the trailing edge and leading edge flaps, respectively. 

The position vector of the mass center is given as
$$
\mathbf{r}_{c.m.} = h\hat{n}_1 - x_\theta\hat{b}_2,
$$
and therefore the corresponding velocity of point $C$ is 
$$
\mathbf{\dot{r}}_{c.m.} = \dot{h}\hat{n}_1 + x_\theta \dot{\theta}\hat{b}_1.
$$
The rotation matrix for transformation between inertial frame of reference to body fixed frame of reference is 
$$
\begin{bmatrix}
\hat{b}_1\\
\hat{b}_2
\end{bmatrix}
=
\begin{bmatrix}
\cos \theta & \sin \theta \\
-\sin \theta & \cos \theta
\end{bmatrix}
\begin{bmatrix}
\hat{n}_1\\
\hat{n}_2
\end{bmatrix}.
$$
The kinetic energy is computed to be
$$
T = \frac{1}{2}m(\mathbf{r}_{c.m.}.\mathbf{r}_{c.m.} ) + \frac{1}{2} I_\theta {\dot{\theta}^2},
$$
$$
T=\frac{1}{2}m(\dot{h}^2+x_\theta^2 \dot{\theta}^2 +2 x_\theta \dot{h} \dot{\theta} \cos{\theta}) + \frac{1}{2}I_\theta {\dot{\theta}^2},
$$
and the corresponding potential energy is 
$$
V=\frac{1}{2}K_h h^2 + \frac{1}{2}K_\theta \theta ^2. 
$$
therefore we can write Lagrangian as $L=T-V$. We apply Euler-Lagrange equations to write the equation of motion as follows

\begin{multline}
\begin{bmatrix}
m & m x_\theta \cos{\theta} \\
m x_\theta \cos{\theta} & m x_\theta^2 + J
\end{bmatrix}
\begin{bmatrix}
\ddot{h} \\
\ddot{\theta}
\end{bmatrix}
+
\begin{bmatrix}
0 & -m x_\theta \dot{\theta}\sin{\theta} \\
0 & 0
\end{bmatrix}
\begin{bmatrix}
\dot{h} \\
\dot{\theta}
\end{bmatrix}
+
\begin{bmatrix}
K_h & 0 \\
0 & K_{\theta}
\end{bmatrix}
\begin{bmatrix}
h \\
\theta
\end{bmatrix}\\
=
\begin{bmatrix}
L(\theta,\dot{\theta}) \cos{\theta} \\
M(\theta,\dot{\theta}) + x_a L(\theta,\dot{\theta})
\end{bmatrix}
+
\begin{bmatrix}
-L_{\beta_1} \cos {(\theta + \beta_1)} -L_{\beta_2} \cos {(\theta + \beta_2}) \\
-L_{\beta_1} (e_1 + d_1 \cos{\beta_1})+L_{\beta_2} (e_2 + d_2 \cos{\beta_2}).
\end{bmatrix}
\end{multline}
The above equation is written in the form of a standard robotic equations of motion $
M(q(t))\ddot{q}(t)+C(q(t),\dot{q}(t))\dot{q}(t)+K(q(t))=Q_a (t) + \tau(t)$, where $q = [h \: \theta]^T $. We have discussed control applications for such systems in detail in Section \ref{sec:intro}. In addition, we employ a simplified version of this equation to validate our online identification and adaptive control strategy in Section \ref{sec:numerics2}.


\begin{thebibliography}{99}
\bibitem{bsb2013}
Joseph W. Bahlman, Sharon M. Swartz1, and Kenneth S. Breuer, {\em Design and Characterization of a Multi-Articulated Robotic Bat Wing}, {\em Bioinspiration \& Biomimetics}, Vol. 8, No. 1, pp. 1-17, 2013. 

\bibitem{bk1989}
H.T. Banks and K. Kunisch, {\em Estimation Techniques for Distributed Parameter Systems,} Birkhauser, Boston, 1989.

\bibitem{bsdr1997}
J. Baumeister, W. Scondo, M.A. Demetriou, and I.G. Rosen,
{\em On-Line Parameter Estimation for Infinite Dimensional Dynamical Systems}, {\em SIAM J. Control. Optim.}, Vol. 35, No. 2, pp. 678-713, 1997.

\bibitem{bs1996}
Martin Brokate and Jurgen Sprekels, {\em Hysteresis and Phase Transitions},  Springer-Verlag, 1996.

\bibitem{corduneaunu2008}
C. Corduneaunu, {\em Integral Equations and Applications},
Cambridge University Press, Cambridge, 2008.

\bibitem{dahmen1996}
Wolfgang Dahmen, {Stability of Multiscale Transformations}, {\em  J. Fourier Anal. Appl.}, Vol. 4, 1996, 341-362.

\bibitem{d1993}
M.A. Demetriou, {Adaptive Parameter Estimation of Abstract Parabolic and Hyperbolic Distributed Parameter Systems,}
{\em Ph.D. thesis,} Departments of Electrical-Systems and Mathematics, University of Southern California, Los Angeles, CA, 1993.

\bibitem{dr1994}
M.A. Demetriou and I.G. Rosen, {\em Adaptive Identification of Second Order Distributed Parameter Systems,} {\em Inverse Problems}, Vol. 10, pp. 261-294, 1994.

\bibitem{dr1994pe}
M.A. Demetriou and I.G. Rosen, {\em On the Persistence of Excitation in the Adaptive Identification of Distributed Parameter Systems,} {\em IEEE Trans. Automat. Control,} Vol. 39, pp. 1117-1123, 1994.

\bibitem{devore1998}
Ronald A. DeVore, {\em Nonlinear Approximation},
{\em Acta Numerica}, Vol. 7, pp. 51-150, January, 1998.

\bibitem{dl1993}
Ronald A. DeVore and George Lorentz,{\em Constructive Approximation}, Springer-Verlag, 1993.

\bibitem{driver1962}
Rodney D. Driver, {Existence and Stability of Solutions of a Delay-Differential System}, {\em Archive for Rational Mechanics and Analysis}, Vol. 10, No. 1, pp. 401-426, January, 1962.

\bibitem{dmp1994}
T.E. Duncan, B. Maslowski, and B. Pasik-Duncan, {\em Adaptive Boundary and Point Control of Linear Stochastic Distributed Parameter Systems,} {\em SIAM J. Control Optim.}, Vol. 32, pp. 648-672, 1994.

\bibitem{ilr2010}
Achim Ilchmann, Hartmut Logemann, and Eugene P. Ryan,
{\em Tracking with Prescribed Transient Performance for Hysteretic Systems}, {\em SIAM J. Control Optim.}, Vol. 48, No. 7, pp. 4731--4752, 2010. 

\bibitem{kp89}
M.A. Krasnoselskii and A.V. Pokrovskii, \emph{Systems with Hysteresis}, Springer, Berlin, 1989.

\bibitem{irs2002}
A. Ilchmann, E.P. Ryan, C.J. Sangwin, {\em Systems of Controlled Functional Differential Equations and Adaptive Tracking}, {\em SIAM J. Control. Optim.}, Vol. 40, No. 6, pp. 1746-1764, 2002. 

\bibitem{khalil}
H. K. Khalil, Nonlinear Systems, 3rd ed. Upper Saddle River, NJ:
Prentice-Hall, 2003

\bibitem{is2012}
Petros Ioannou and Jing Sun, {\em Robust Adaptive Control},
Dover, 2012.

\bibitem{keinert2004}
Fritz Keinert, {\em Wavelets and Multiwavelets}, Chapman \& Hall, CRC Press, 2004.

\bibitem{k1956a}
N.N. Krasovskii, {\em On the Application of the Second Method of A.M. Lyapunov to Equations with Time Delays}, {\em Prikl. Mat. i Mekh.}, Vol. 20, pp. 315-327, 1956.

\bibitem{k1956b}
N.N. Krasovskii, {On the Asymptotic Stability of Systems with After-Effect,} {\em Prikl. Mat. i Mekh.}, 20, 513--518, 1956.


\bibitem{lda2004}
Frank L. Lewis, Darren M. Dawson, and Chaouki T. Abdallah,
{\em Robot Manipulator Control: Theory and Practice},
Marcel Dekker, Inc., 2004. 

\bibitem{m1949}
A.D. Myshkis, {\em General Theory of Differential Equations with a Retarded Argument,} {\em Uspekhi Mat. Nauk}, Vol. 4, No. 5(33), pp. 99--141, 1949.


\bibitem{na2005}
Kumpati S. Narendra and Anuradha M. Annaswamy, {\em Stable Adaptive Systems}, Dover, 2005.

\bibitem{rudakov1978}
V.P. Rudakov, {\em Qualitative Theory in a Banach Space, Lyapunov-Krasovskii Functionals, and Generalization of Certain Problems}, {\em Ukrainskii Matematicheskii Zhurnal}, Vol. 30, No. 1, pp. 130-133, January-February, 1978. 

\bibitem{rudakov1974}
V.P. Rudakov, {\em On Necessary and Sufficient Conditions for the Extendability of Solutions of Functional-Differential Equations of the Retarded Type},
{\em Ukr. Mat. Zh.}, Vol. 26, No. 6, pp. 822-827, 1974.

\bibitem{rs2002}
E.P. Ryan and C.J. Sangwin, {\em Controlled Functional Differential Equations and Adaptive Tracking}, 
{\em Systems and Control Letters}, Vol. 47, pp. 365--374, 2002. 
\bibitem{sb2012}
Shankar Sastry and Marc Bodson, {\em Adaptive Control: Stability, Convergence and Robustness}, Dover, 2011.

\bibitem{ssvo2010}
Bruno Siciliano, Lorenzo Sciavicco, Luigi Villani, and Guiseppe Oriolo, {\em Robotics:Modeling, Planning, and Control}, Springer-Verlag, London, 2010.

\bibitem{shv2005}
Mark W. Spong, Seth Hutchinson, and M. Vidyasagar, 
{\em Robot Modeling and Control}, Wiley, 2005

\bibitem{visintin1994}
Augusto Visintin, {\em Differential Models of Hysteresis},
Springer, 1994.

\bibitem{ksw2003}
Kurdila, A.J., Li, J., Strganac, T.W., and Webb, G., {\em Nonlinear Control Methodologies for Hysteresis in PZT Actuated On-Blade Elevons}, Journal of Aerospace Engineering, Volume 16, Issue 4, pp. 167-176, October 2003.



\bibitem{VT2013}
K. Viswanath and D. Tafti, {\em Effect of Stroke Deviation on Forward Flapping Flight},AIAA Journal, pp. 145-160, Vol. 51, No. 1, January, 2013.

\bibitem{GT2010}
P. Gopalakrishnan and D. Tafti, {\em  Effect of Wing Flexibility on Lift and Thrust Production in Flapping Flight}, AIAA Journal, pp. 865-877,
  Vol. 48, No. 5, May, 2010.
  
\bibitem{Dadashi2016}  
S. Dadashi, J. Feaster, J. Bayandor, F. Battaglia, A. J. Kurdila, {\em Identification and adaptive control of history dependent unsteady aerodynamics for a flapping insect wing}, Nonlinear Dyn (2016) 85: 1405. 

\bibitem{tobak}
Tobak, Murray, and Lewis B. Schiff. {\em On the formulation of the aerodynamic characteristics in aircraft dynamics.} National Aeronautics and Space Administration, 1976.

\bibitem{tavernini}
Tavernini L, {\em Linear multistep methods for the numerical solution of volterra functional differential equations} Applicable Analysis 3(1973), 169-185.

\bibitem{Ko_Kurdila1997}
Jeonghwan Ko, Andrew J. Kurdila, and Thomas W. Strganac.  {\em Nonlinear Control of a Prototypical Wing Section with Torsional Nonlinearity}, Journal of Guidance, Control, and Dynamics, Vol. 20, No. 6 (1997), pp. 1181-1189.

\bibitem{dadashi2016_CDC}
Shirin Dadashi, Parag Bobade, Andrew J. Kurdila, {\em Error Estimates for Multiwavelet Approximations of a Class of History Dependent Operators } 2016 IEEE 55th Conference on Decision and Control (CDC)
\end{thebibliography}
\end{document}